\documentclass[a4paper,11pt]{amsart}
\usepackage[foot]{amsaddr}
\title{The comonotone flow of a stochastically monotone Feller process on the real line}
\email{jberard@unistra.fr / brieuc.frenais@univ-lorraine.fr}
\address{Institut de Recherche Math\'ematique Avanc\'ee, UMR 7501, Universit\'e de Strasbourg and Institut Elie Cartan de Lorraine, UMR 7502, Université de Lorraine}
\author{Jean B\'erard and Brieuc Fr\'enais}

\usepackage{ae,lmodern}
\usepackage[utf8]{inputenc}
\usepackage[T1]{fontenc}
\usepackage{fullpage}

\usepackage{enumitem}

\usepackage{amsmath}			
\usepackage{amssymb}			
\usepackage{mathrsfs}
\usepackage{stix}
\usepackage{bbm}
\usepackage{graphicx}
\usepackage{subcaption}

\usepackage{stmaryrd}			

\usepackage{imakeidx}

\usepackage{hyperref}	

\usepackage{xcolor}
\usepackage{blkarray}

\newtheorem{proposition}{Proposition}
\newtheorem{theorem}{Theorem}
\newtheorem{lemma}{Lemma}
\newtheorem{corollary}{Corollary}

\newtheorem{remark}{Remark}

\newcommand{\BB}{\mathcal{B}}

\newcommand{\CC}{\mathcal{C}}

\newcommand{\E}{\mathbb{E}}

\renewcommand{\i}{\mathrm{i}}

\renewcommand{\P}{\mathbb{P}}

\newcommand{\Q}{\mathbb{Q}}

\newcommand{\R}{\mathbb{R}}

\newcommand{\s}{\sigma}

\newcommand{\Z}{\mathbb{Z}}

\newcommand{\SM}{\mathcal{SM}}
\newcommand{\Lip}{\mbox{Lip}}
\newcommand{\sLip}{\mbox{\scriptsize Lip}}

\newcommand\sm{\leq_{\mathbf{sm}}}
\newcommand\st{\leq_{\mathbf{st}}}

\renewcommand{\bar}{\overline}

\newcommand{\ent}[2]{\llbracket #1,#2\rrbracket}

\renewcommand{\hat}{\widehat}
\renewcommand{\check}{\widecheck}

\newcommand{\n}[2]{\left\|#1\right\|_{#2}}

\renewcommand{\Tilde}{\widetilde}

\newcommand{\thref}[1]{\hyperref[#1]{Theorem \ref{#1}}}
\newcommand{\defiref}[1]{\hyperref[#1]{Definition \ref{#1}}}
\newcommand{\figref}[1]{\hyperref[#1]{Figure \ref{#1}}}
\newcommand{\propref}[1]{\hyperref[#1]{Proposition \ref{#1}}}
\renewcommand{\eqref}[1]{\hyperref[#1]{(\ref{#1})}}
\newcommand{\chapref}[1]{\hyperref[#1]{Chapter \ref{#1}}}
\newcommand{\lemref}[1]{\hyperref[#1]{Lemma \ref{#1}}}
\newcommand{\cororef}[1]{\hyperref[#1]{Corollary \ref{#1}}}
\newcommand{\secref}[1]{\hyperref[#1]{Section \ref{#1}}}

\newcommand{\un}{\mathbf{1}}

\begin{document}

\maketitle

\begin{abstract}
  We show that any stochastically monotone Feller semigroup on $\mathbb R$ can be extended by a consistent family of order-preserving Feller semigroups on the successive powers of $\mathbb R$. We exhibit a specific such family, which is uniquely characterized by a maximality property with respect to the super-modular order on $\mathbb R^n$. A consequence is that, in this fairly general setting, there always exists a coupling between $n$ càdlàg versions of the underlying Markov process starting from $n$ distinct initial positions, which do not cross one another.
\end{abstract}

\section{Introduction}

Consider a Markov semigroup $(P_t)_{t \in \R_+}$ whose state space is the real line $\R$, with the mild additional regularity assumption that $(P_t)_{t \in \R_+}$ enjoys the Feller property (see Sections \ref{ss:Markov-kernels} and \ref{ss:Feller-semigroup} below for formal definitions). For any $x \in \R$, one can then define a collection of real-valued random variables $(X_t^x)_{t \in \R_+}$ that forms a Markov process with càdlàg paths, starting at $x$, and governed by the semigroup $(P_t)_{t \in \R_+}$. Now ask the following sequence of questions: 
\begin{itemize}
\item[(a)] Given an integer $n \geq 2$, and $n$ real numbers $x_1,\ldots,x_n$, is it possible to define a collection of random variables $(X_t^{x_1},\ldots,X_t^{x_n})_{t \in \R_+}$ such that, for every $1 \leq i \leq n$, $(X_t^{x_i})_{t \in \R_+}$ forms a Markov process with càdlàg paths, starting at $x_i$, and governed by the semigroup $(P_t)_{t \in \R_+}$, which is {\bf order-preserving},
i.e. such that whenever $x_i \leq x_j$, one also has that $X_t^{x_i} \leq X_t^{x_j}$ for all $t$ ?

\item[(b)] In addition, can one require $(X_t^{x_1},\ldots,X_t^{x_n})_{t \in \R_+}$ to form a Markov process on $\R^n$ governed by a Feller semigroup ? 

\item[(c)] Furthermore, 
  can these joint evolutions be defined in a consistent way for the various values of $n$, i.e. in such a way that, for $k \leq n$, and $1 \leq i_1,\ldots,i_k \leq n$, the projection onto the $k$ coordinates $i_1,\ldots,i_k$ of the $n$ paths $(X_t^{x_1},\ldots,X_t^{x_n})_{t \in \R_+}$ has the same probability distribution (as a stochastic process in $\R^k$) as the $k$ paths $(X_t^{y_1},\ldots,X_t^{y_k})_{t \in \R_+}$ starting from $y_1 = x_{i_1},\ldots,y_k = x_{i_k}$ (see Section \ref{ss:consistent} for a formal definition) ?

    
\end{itemize}

It is easily seen that a necessary condition for a positive answer to (a) is that the semigroup $(P_t)_{t \in \R_+}$ is stochastically monotone (see Section \ref{sec:stochastic-orders} for the definition). Our first result (Theorem \ref{th:extension}) is that stochastic monotonicity is also a sufficient condition for a positive answer to (a)-(b)-(c). 
Our second result (Theorem \ref{th:caract-sm}) shows that the construction used in the proof of Theorem \ref{th:extension} leads to a uniquely characterized object, that we call the comonotone flow associated with the Feller semigroup $(P_t)_{t \in \R_+}$. The characterization involves an optimality property with respect to the supermodular stochastic order (see Section \ref{sec:stochastic-orders} for the definition). 

In the rest of this introduction, we start by recalling basic definitions about Markov and Feller semigroups (Sections \ref{ss:Markov-kernels} and \ref{ss:Feller-semigroup}), the property of consistency (Section \ref{ss:consistent}), and stochastic orders (Section \ref{sec:stochastic-orders}). We then state our main results in Section \ref{ss:statement}, and discuss motivations, related results, and possible extensions in Section \ref{ss:discuss}.

\subsection{Markov kernels}\label{ss:Markov-kernels}

A Markov kernel $K$ on a measurable space $(S,\mathscr{S})$ is a map from $S \times \mathscr{S}$ to $\R$ such that 
\begin{itemize}
\item[(i)] for all $x \in S$, $K(x,\cdot)$ is a probability measure on $(S,\mathscr{S})$, 
\item[(ii)] for all $B \in \mathscr{S}$, $K(\cdot,B)$ is a measurable real-valued function on $S$.
\end{itemize}
 When there is no ambiguity regarding the choice of the $\sigma$-algebra $\mathscr{S}$, we simply say that $K$ is a Markov kernel on $S$. 
Given a bounded measurable real-valued function $f$ on $S$, we define a bounded measurable real-valued function $Kf$ on $S$ by $Kf(x) = \int_S f(y) dK(x,y)$. Denoting by $\BB_b(S)$ the vector space of bounded real-valued measurable functions on $S$, equipped with the sup-norm $\n{f}{\infty} = \sup\limits_{x \in S} |f(x)|$, the map $f \mapsto K f$ defines a linear operator from $\BB_b(S)$ into itself, and satisfies \begin{itemize}
\item[($\alpha$)] $\n{Kf}{\infty} \leq \n{f}{\infty}$,   
\item[($\beta$)] $K f \geq 0$ when $f \geq 0$,
\item[($\gamma$)] $K \un = \un$ (where $\un$ 
denotes the constant function equal to $1$).
\end{itemize}

\subsection{Markov and Feller semigroups}\label{ss:Feller-semigroup}

Given two Markov kernels $K,L$ on $S$, the composition of the two kernels is yet another Markov kernel $KL$ defined by $(KL)(x,B) = \int_{S} L(y,B) dK(x,y)$. The composition is an associative (but in general non-commutative) operation on Markov kernels. Moreover, for $f \in \BB_b(S)$, we have that $(KL)f = K (Lf)$. We say that a family of Markov kernels $(K_t)_{t \in \R_+}$ on $S$ is a {\it Markov semigroup} if 
\begin{itemize}
\item[(I)] for all $x \in S$ and $B \in \mathscr{S}$, $K_0(x,B)=\delta_x(B)$, 
\item[(II)] the Chapman-Kolmogorov equation holds: for all $s,t \geq 0$,  $K_{s+t} =  K_s K_t$.
\end{itemize}
Moreover, we say that a family $(X_t)_{t \in \R_+}$ of $S$-valued random variables defined on the same probability space 
is a Markov process governed by the Markov semigroup $(K_t)_{t \in \R_+}$ when, for all $s,t \in \R_+$, and all $B \in \mathscr{S}$, one has
$$\E \left(\un_B(X_{s+t})| \sigma(X_u: \ u \in [0,s]) \right) \overset{\mbox{\scriptsize a.s.}}{=} K_t(X_s,B).$$

Now assume that $S$ is a locally compact separable metric space (abbreviated {\it lcsm} in the sequel), equipped with the corresponding Borel $\sigma$-algebra. We denote by $\CC_0(S)$ the vector space of continuous real-valued functions on $S$ vanishing at infinity. Note that $\CC_0(S)$ equipped with the sup-norm is a Banach space.

We say that a Markov semigroup $(K_t)_{t \in \R_+}$ on $S$ enjoys the {\it Feller property} (or, more succintly, that it is a {\it Feller semigroup}) when: 
\begin{itemize}
\item[(Fa)] $\forall t \geq 0, \ \forall f\in \CC_0(S), \ K_tf \in \CC_0(S)$,
\item[(Fb)] $\forall f\in \CC_0(S), \ \lim_{t \to 0+} \n{K_t f - f }{\infty} = 0$.  
\end{itemize}
Let us now denote by $\CC_b(S)$ the vector space of bounded continuous real-valued functions on $S$, and observe that property (Fa) implies\footnote{See Proposition \ref{p:continue} below.}, but is not in general equivalent\footnote{See e.g. \cite{Bez+22}.} to
\begin{itemize}
\item[(Fa')] $\forall t \geq 0, \ \forall f\in \CC_b(S), \ K_tf \in \CC_b(S)$.
\end{itemize}
 Moreover, (Fb) may be replaced by the apparently weaker assumption of pointwise (instead of uniform) convergence
\begin{itemize}
\item[(Fb')] $\forall f\in \CC_0(S), \forall x \in S, \ \lim_{t \to 0+} K_t f(x) = f(x)$, 
\end{itemize}
but it turns out\footnote{See e.g. \cite{Kal02} Theorem 19.6, or \cite{RevYor99} Chapter III, Proposition (2.4).} that a Markov semigroup satisfying (Fa) and (Fb') also satisfies (Fb). 

\subsection{Consistent families}\label{ss:consistent}

For every integer $n \geq 2$, equip the product space $S^n$ with the product $\sigma$-algebra $\mathscr{S}^{\otimes n}$. Given integers $1 \leq k \leq n$, and $i_1,\ldots, i_k \in \ent{1}{n}$, we denote by $\pi^n_{i_1,\ldots,i_k}$ the projection from $S^n$ to $S^k$ defined by $\pi^n_{i_1,\ldots,i_k}(x_1,\ldots, x_n)=(x_{i_1},\ldots, x_{i_k})$. 
We say that a family of Markov kernels $(K^{(n)})_{n \geq 1}$, where, for each integer $n$, $K^{(n)}$ is a Markov kernel on $S^n$, is a {\it consistent family} if, for all $1 \leq k \leq n$,  all $i_1,\ldots, i_k \in \ent{1}{n}$, all $\mathbf{x} \in S^n$ and all measurable subset $B$ of $S^k$, 
\begin{equation}\label{e:consistent}K^{(n)}\left(\mathbf{x},(\pi^n_{i_1,\ldots, i_k})^{-1}(B)\right) = K^{(k)}\left(\pi^n_{i_1,\ldots, i_k}(\mathbf{x}), B\right).\end{equation}
We say that a family of Markov kernels $(K^{(n)})_{n \geq 2}$ is a {\it consistent extension} of a Markov kernel $K$ on $S$, if the family $(K^{(n)})_{n \geq 1}$, with $K^{(1)}=K$, is a consistent family. A family of Markov semigroups $\left((K^{(n)}_t)_{t \in \R_+} \right)_{n \geq 1}$ is said to be consistent if, for every $t \in \R_+$, $(K^{(n)}_t)_{n \geq 1}$ is a consistent family of Markov kernels; a consistent extension of a Markov semigroup on $S$ is defined accordingly.

Note that equal indices are allowed in \eqref{e:consistent}, which implies in particular that, for all $x \in S$, we have \begin{equation}\label{e:egalite-paires}K^{(2)}\left((x,x),\{ (y,y):  y \in S \}\right)=1.\end{equation}

\subsection{Stochastic orders}\label{sec:stochastic-orders}

Now assume that $S = \R$, and denote by $\BB_b^{\nearrow}(\R)$ the set of bounded non-decreasing real-valued functions on $\R$. The usual stochastic order $\mu \st \nu$ between (Borel) probability measures on $\R$ (see \cite{MulSto02}) is defined by the fact that, for all $f \in \BB_b^{\nearrow}(\R)$, one has $\int_{\R} f(x) d \mu(x) \leq \int_{\R} f(x) d \nu(x)$. We say that a Markov semigroup $(K_t)_{t \in \R_+}$ on $\R$ is {\it stochastically monotone} when: 
 \begin{itemize}
\item[(M)] $\forall t \geq 0, \ \forall f\in\BB_b^{\nearrow}(\R) , \ K_tf \in \BB_b^{\nearrow}(\R)$.
\end{itemize}
An immediately equivalent formulation of (M) in terms of $\st$ is that 
\begin{equation}\label{e:stoch-monot}\forall x,y \in \R, \  x \leq y \Rightarrow K_t(x,\cdot) \st K_t(y,\cdot).\end{equation} 

Given $\mathbf{x}=(x_1,\ldots, x_n) \in \R^n$ and $\mathbf{y}=(y_1,\ldots, y_n) \in \R^n$, we say that $\mathbf{y}$ is order-compatible with $\mathbf{x}$ if, for all $1 \leq i,j \leq n$, $x_i \leq x_j  \Rightarrow y_i \leq y_j$. We then denote $$\R^n_{\mathbf{x}} = \{ \mathbf{y} \in \R^n \mbox{ such that $\mathbf{y}$ is order-compatible with $\mathbf{x}$} \},$$ and we say that a Markov kernel $K^{(n)}$ on $\R^n$ is {\it order-preserving} when, for all $\mathbf{x} \in \R^n$, we have that $K^{(n)}\left(\mathbf{x}, \R^n_{\mathbf{x}} \right) = 1$.

Denote by $\SM_b(\R^n)$ the set of real-valued bounded Borel super-modular functions on $\R^n$, i.e. bounded Borel functions $f: \R^n \to \R$ such that, for all $\mathbf{x},\mathbf{y} \in \R^n$, $f(\mathbf{x})+ f(\mathbf{y}) \leq  f(\mathbf{x} \vee \mathbf{y}) + f(\mathbf{x} \wedge \mathbf{y})$, where, for $\mathbf{x}=(x_1,\ldots,x_n)$ and  $\mathbf{y}=(y_1,\ldots,y_n)$, 
we set $\mathbf{x} \vee \mathbf{y} = (\max(x_1,y_1),\ldots,\max(x_n,y_n))$ and $\mathbf{x} \wedge \mathbf{y} = (\min(x_1,y_1),\ldots,\min(x_n,y_n))$.

Given $n \geq 2$ and two (Borel) probability measures $\mu,\nu$ on $\R^n$, we say that $\mu \sm \nu$ (see e.g. \cite{MulSto02}) when
\begin{equation}\label{e:supermodular} \forall f \in \SM_b(\R^n) ,    \  \int_{\R^n} f(\mathbf{x}) d\mu(\mathbf{x}) \leq  \int_{\R^n} f(\mathbf{x}) d\nu(\mathbf{x}).\end{equation}
Note that $\sm$ defines a partial order on the set of probability measures on $\R^n$, and that two probability measures that are comparable with respect to $\sm$ must have the same one-dimensional marginal distributions.

\subsection{Statement of the main results}\label{ss:statement}

\begin{theorem}\label{th:extension}
If $(P_t)_{t \in \R_+}$ is a stochastically monotone Feller semigroup on $\R$, there exists a consistent extension of $(P_t)_{t\in\R_+}$ by a family $\left((P^{(n)}_t)_{t \in \R_+} \right)_{n \geq 1}$
 of order-preserving Feller semigroups.
\end{theorem}
For $n \geq 2$, denote by $\mathfrak{M}_{n}$ the set of Feller semigroups $(M_t^{(n)})_{t \in \R_+}$ on $\R^n$ such that, for every $t \geq 0$, every $\mathbf{x}=(x_1,\ldots, x_n) \in \R^n$, and every $1 \leq i \leq n$, $M_t^n(\mathbf{x},(\pi^n_i)^{-1}(B)) = P_t(x_i,B)$.

\begin{theorem}\label{th:caract-sm}
Let $(P_t)_{t \in \R_+}$ be a stochastically monotone Feller semigroup on $\R$, and, for $n \geq 2$, denote by $(\breve{P}_t^{(n)})_{t \in \R_+}$ the\footnote{A priori, this may not be unique, but it turns out that it is in view of the present statement.} Markov semigroup constructed in the proof of Theorem \ref{th:extension}. Then, for every  $(M_t^{(n)})_{t \in \R_+} \in \mathfrak{M}_{n}$, $s \in \R_+$ and $\mathbf{x} \in \R^n$,
one has that $M_s^{(n)}(\mathbf{x},\cdot) \sm \breve{P}^{(n)}_s(\mathbf{x},\cdot)$. Since by construction $(\breve{P}^{(n)}_t)_{t \in \R_+} \in \mathfrak{M}_{n}$, and since $\sm$ is a partial order, this shows that $(\breve{P}_t^{(n)})_{t \in \R_+}$ is uniquely characterized as the solution of the following maximization problem with respect to $\sm$: 
$$\breve{P}^{(n)}_s(\mathbf{x},\cdot)  = \max_{(M_t^{(n)})_{t \in \R_+} \in \mathfrak{M}_{n}} M_s^{(n)}(\mathbf{x},\cdot).$$
\end{theorem}

\subsection{Discussion}\label{ss:discuss}

Broadly speaking, Theorem \ref{th:extension} is a monotonicity equivalence result in the sense of \cite{FilMac01}: starting from a monotonicity property within a family of probability measures, one deduces the existence of an order-preserving coupling, i.e. a family of random variables providing an effective realization of the monotonicity property. The archetype for such results is Strassen's theorem\footnote{Strassen's theorem is valid in the general context of probability measures on Polish spaces. See \cite{Str65} for Strassen's original paper, and \cite{Lin99} for useful additional elements.}: if $\mu \st \nu$, there exists a pair of random variables $(X,Y)$ such that $\mbox{Law}(X)=\mu$, $\mbox{Law}(Y)=\nu$, and almost surely $X \leq Y$. In our context, the existence of the Feller semigroup $(P_t^{(n)})_{t \in \R_+}$ shows that\footnote{From classical results (see e.g. \cite{RevYor99} or \cite{Kal02}), we have the existence of an $\R^n$-valued Markov process  $(X_t^{x_1},\ldots, X_t^{x_n})_{t \in \R_+}$ with càdlàg paths, governed by  $(P_t^{(n)})_{t \in \R_+}$ and starting at $\mathbf{x}$. Then, by the order-preserving property of $P_t^{(n)}$, almost surely, for all $t \in \Q_+$,  $(X_t^{x_1},\ldots, X_t^{x_n}) \in  \R^n_{\mathbf{x}}$. Then, using right-continuity of paths and the fact that $ \R^n_{\mathbf{x}}$ is a closed subet of $\R^n$, we have that, almost surely, for all $t \in \R_+$, $(X_t^{x_1},\ldots, X_t^{x_n}) \in  \R^n_{\mathbf{x}}$.}, given $\mathbf{x}=(x_1,\ldots, x_n) \in \R^n$, one can define a family of real-valued random variables $(X_t^{x_1},\ldots, X_t^{x_n})_{t \in \R_+}$ on the same probability space, such that:
\begin{itemize}
\item for $1 \leq i \leq n$, $X_t^{x_i}$ is a real-valued Markov process with càdlàg paths, starting at $x_i$, and governed by $(P_t)_{t \in \R_+}$, 
\item for all $t \in \R_+$, $(X_t^{x_1},\ldots, X_t^{x_n})$ is order-compatible with $(x_1,\ldots, x_n)$.
\end{itemize}
The fact that $(X_t^{x_1},\ldots, X_t^{x_n})_{t \in \R_+}$ is itself a Markov process on $\R^n$ governed by a Feller semigroup comes as an important additional property when such order-preserving constructions are to be used in the context of interacting particle systems. For an example, we refer to \cite{BerFre23a}, where Theorem \ref{th:extension} plays a key role in the definition of the coupling that is used to control the convergence of a branching-selection particle system to its hydrodynamic limit; this was indeed our original motivation for investigating the question leading to the present work.

Existence results for order-preserving couplings of the same kind as Theorem \ref{th:extension} have been established in a variety of contexts, including\footnote{For the sake of brevity, we limit ourselves to a small sample of references, and point the interested reader to the works quoted by, or quoting, these references for a more exhaustive view of this subject.} discrete-time Markov chains on partially ordered Polish spaces (\cite{Kam+77}), continuous-time Markov chains on countable partially ordered spaces (\cite{Lop+00}), jump processes on partially ordered Polish spaces (\cite{Zha00}), interacting particle systems on $\{0,1\}^{\Z^d}$ (\cite{Lig05}). 

In the setting of continuous-time Markov processes on the real line, let us mention two important special cases for which the conclusion of Theorem \ref{th:extension} can be established relatively easily:  
\begin{itemize}
\item Lévy processes (see e.g. \cite{Ber96}): using the \textit{parallel coupling}, obtained by setting, for all $x \in \R$ and $t \geq 0$, $X_t^x = x + X_t^0$, where $(X_t^0)_{t\in\R_+}$ is a version of the Lévy process starting at $0$, see Figure \ref{fig:parallel-coupling};
\item Feller processes with continuous paths: one defines $P_t^{(n)}$ as the distribution at time $t$ of a family\footnote{A construction for general Feller processes (not necessarily with continous paths) leading to a strong Markov process is given in \cite{Eva+13} Section 2.1. Continuity of paths easily shows that $P_t^{(n)}$ is order-preserving. The Feller property of $(P_t^{(n)})_{t\in\R_+}$ can then be established using the fact that we have a consistent extension of a Feller semigroup on the real line by order-preserving semigroups, exactly as in the Proof of Theorem \ref{th:extension}.} of $n$ trajectories that evolve independently until they meet, and stick together thereafter (this is often called the \textit{Doeblin coupling}), see 
Figure \ref{fig:doeblin-coupling}; the construction with $n=2$ is a nice way of proving the not-so-obvious fact that any Feller process on the real line with continuous paths is stochastically monotone.
\end{itemize}
 \begin{figure}
    \centering
       \begin{subfigure}{.49\textwidth}
        \centering
        \includegraphics[width=\textwidth]{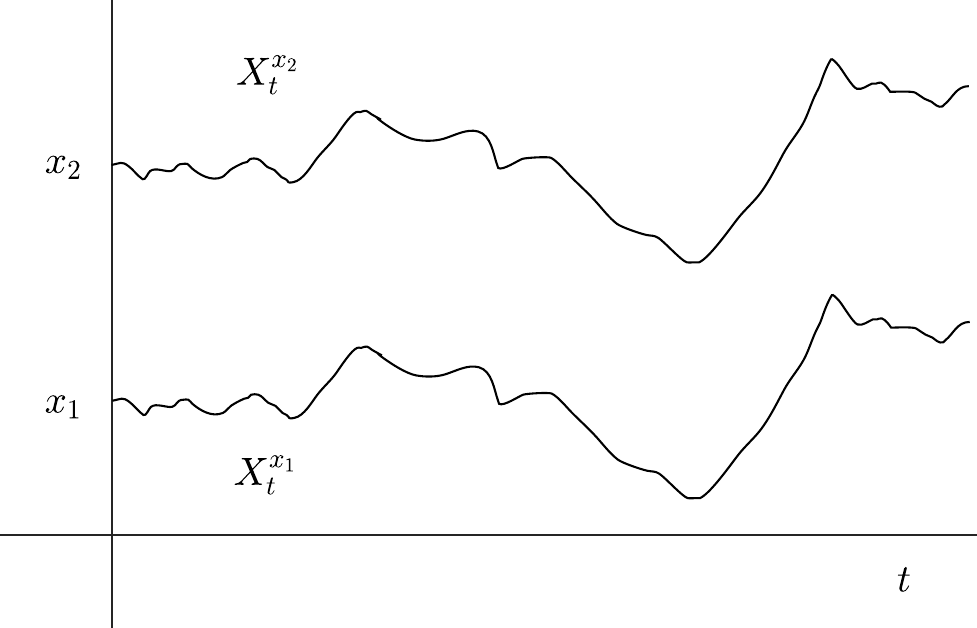}
        \caption{Parallel coupling}
        \label{fig:parallel-coupling}
    \end{subfigure}
    \begin{subfigure}{.49\textwidth}
        \centering
        \includegraphics[width=\textwidth]{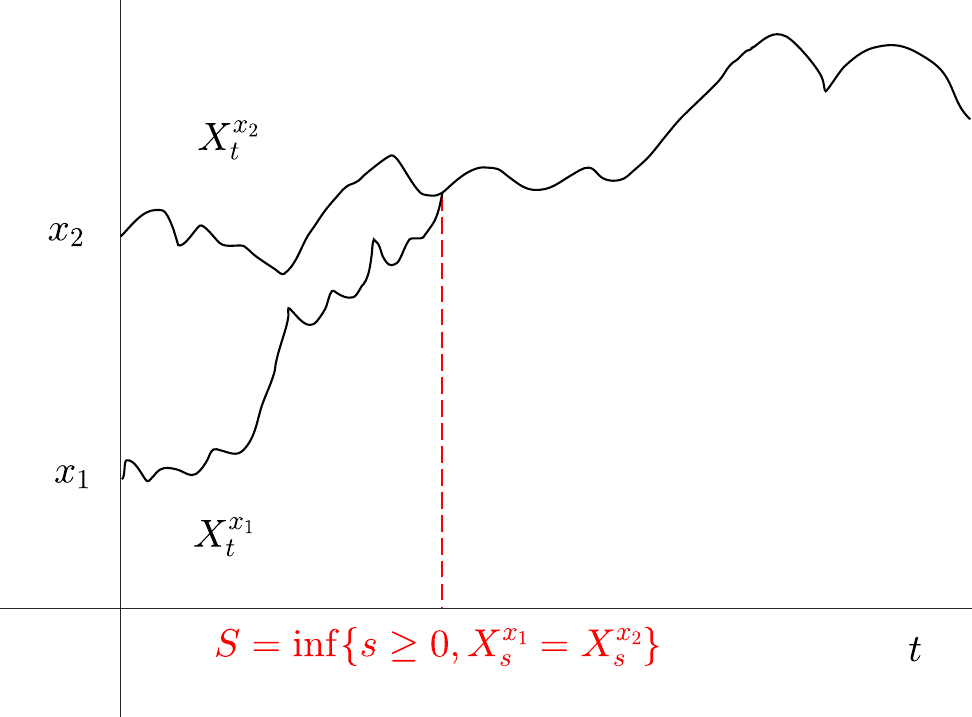}
        \caption{Doeblin coupling}
        \label{fig:doeblin-coupling}
    \end{subfigure}
    \caption{Special cases of Theorem \ref{th:extension} with easier constructions}
    \label{fig:monotone-coupling}
\end{figure}
In the case of jump processes (\cite{Che04}), and more generally, of Lévy-type processes (\cite{Wan13, Kol11}), necessary and sufficient explicit conditions for stochastic monotonicity, expressed in terms of the infinitesimal generator of the process, have been obtained. One nice feature of Theorem \ref{th:extension} is that it holds in a completely general setting, regardless of the specific structure of the underlying process, as soon as it is stochastically monotone and enjoys the Feller property, which are also necessary conditions for the conclusion of the theorem to hold. 

Theorem \ref{th:caract-sm} shows that, for all $n \geq 2$, $s \in \R_+$ and $\mathbf{x}=(x_1,\ldots,x_n)$, $\breve{P}^{(n)}_s(\mathbf{x},\cdot)$ produces an $n$-dimensional coupling of the distributions $P_s(x_1,\cdot),\ldots,P_s(x_n,\cdot)$, enjoying a maximality property with respect to the super-modular order. Without the constraint that this coupling has to be the distribution at time $s$ of a Feller process on $\R^n$ whose one-dimensional marginals evolve according to $(P_t)_{t \in \R_+}$, the maximum would be achieved by the classical comonotone coupling of $P_s(x_1,\cdot),\ldots,P_s(x_n,\cdot)$ (see \cite{MulSto02}). Loosely speaking, $\breve{P}^{(n)}$ is obtained (as a well-defined limit) by composing infinitely many such comonotone couplings associated with an infinitesimal time interval $s$, so we choose to call the resulting object the {\it comonotone flow} associated with $(P_t)_{t\in\R_+}$. Strictly speaking, we have not defined a flow but only a consistent family $\left((P_t^{(n)})_{t \in \R_+}\right)_{n \geq 1}$ of Feller Markov semigroups. However, the general theory developed in\footnote{See also the correction \cite{LeJRai20} and the fully corrected version arXiv:math/0203221v6. Condition \eqref{e:consistent} and its consequence \eqref{e:egalite-paires} precisely match the definition in \cite{LeJRai04} of a compatible family of Feller semigroups satisfying $P_t^{(2)}f^{\otimes 2}(x,x) = P_t f^2(x)$ for all $x \in S$ and $f \in \CC_0(S)$.} \cite{LeJRai04} shows that this family of semigroups corresponds to a stochastic flow, in a precise sense, and our terminology reflects this property. We believe the comonotone flow thus obtained to be an interesting extension of stochastic flows usually encountered in the context of stochastic differential equations: more specifically, when the underlying process $(X_t)_{t\in\R_+}$ is the solution of a stochastic differential equation (with suitable regularity assumptions on the coefficients), we conjecture that the comonotone flow coincides with the stochastic flow of strong solutions of the SDE obtained with a single Brownian motion (see Kunita \cite{Kun97, Kun19}). We postpone the study of such properties to future work.

\subsection{Organization of the paper}

In Section \ref{s:Feller}, we collect various results related to the Feller property of Markov semigroups, on general (locally compact separable metric, or compact metric) spaces, then on $\R$. Section \ref{s:extension} is devoted to the proof of Theorem \ref{th:extension}, and Section \ref{s:caract-sm} to the proof of Theorem \ref{th:caract-sm}. Section \ref{s:appendix-KamKreOBr} is an appendix where we explain why we think an element is missing in the proof of some results in \cite{Kam+77} which are related to the conclusion of Theorem \ref{th:extension} in the present paper. Section \ref{s:appendix-simulations} shows some numerical illustrations.

\section{Feller-related results}\label{s:Feller}

This section is devoted to results related to the Feller property of Markov semigroups. We could not find a reference containing an exposition of these properties that completely suited our needs, and so decided to give a short self-contained exposition instead. In Section \ref{ss:Feller-general}, we provide alternative formulations of the Feller property that turn out to be more suitable for the arguments developed in Section \ref{s:extension}, namely, transferring the assumed regularity properties of the family of kernels $(P_t)_{t \in \R_+}$ on $\R$, to families of kernels on $\R^n$ that extend $P_t$.  In Section \ref{ss:R-bar}, we show that the combination of stochastic monotonicity with the Feller property allows one to extend $(P_t)_{t \in \R_+}$ to a Feller semigroup $(\Tilde{P}_t)_{t \in \R_+}$ on the extended real line $\overline{\R}$, whose advantage (in our context) over the usual one-point compactification $\R \cup \{ \infty \}$ is its compatibility with the underlying order on $\R$.

\subsection{General properties}\label{ss:Feller-general}

Given a lcsm space, we denote by $\mathcal{P}(S)$ the set of (Borel) probability measures on $S$, equipped with the topology of weak convergence, which is metrizable (see below). To a Markov kernel $K$ on $S$, we associate the map $\check{K}$ from $S$ to $\mathcal{P}(S)$, defined as $x \mapsto K(x,\cdot)$. Conversely, to a map $\mathbf{K}$ from $S$ to $\mathcal{P}(S)$, we associate the map 
$\hat{\mathbf{K}}$ from $S \times \mathscr{S}$ to $\R$ defined as $(x,B) \mapsto \left[ \mathbf{K}(x)\right](B)$, which may or may not be a Markov kernel, depending on whether the map $\hat{\mathbf{K}}(\cdot,B)$ is measurable for every $B$.

\begin{lemma}\label{l:noyau-cont}
Given a lcsm space $S$, and a continuous map $\mathbf{K} \ : \ S \to \mathcal{P}(S)$, $\hat{\mathbf{K}}$ is a Markov kernel.
\end{lemma}
\begin{proof}
Arguing as in the proof of Proposition \ref{p:continue} below (there is no circular reference, we just wanted to avoid writing the same proof twice), we have that, for all $f \in \CC_b(S)$, the map $x \mapsto \hat{\mathbf{K}}f(x) = \int_S f(y) d\left[\mathbf{K}(x)\right](y)$ is itself an element of $\CC_b(S)$, and in particular is a Borel map.

We now apply Theorem 0.2.2 in \cite{RevYor99}, which is a functional version of the monotone class theorem: the set $\mathscr{H}$ of functions $f \in \BB_b(S)$ such that $x \mapsto  \hat{\mathbf{K}}f(x)$ is Borel contains $\CC_b(S)$ -- which plays the role of $\mathscr{C}$ in \cite{RevYor99} -- and the assumptions of the theorem are met\footnote{The set $\mathscr{H}$ is a vector space that contains the constant functions and the supremum of any bounded non-decreasing sequence of its non-negative elements (thanks to the monotone convergence theorem); the set $\mathscr{C}$ is stable under pointwise multiplication.}, so that $\mathscr{H}$ contains all $\s(\CC_b(S))$-measurable functions, and thus all bounded real-valued Borel functions on $S$. In particular $ \hat{\mathbf{K}} \un_B$ is a Borel function for any Borel set $B$ in $S$, so $\hat{\mathbf{K}}$ is indeed a Markov kernel on $S$.
\end{proof}

\begin{proposition}\label{p:continue}
Given a Markov kernel $K$ on a lcsm space $S$, the following three properties are equivalent: 
\item[$\mathbf{(i)}$]  $\check{K}$ is a continuous map,
\item[$\mathbf{(ii)}$ ] $\forall f\in \CC_b(S), \ K f \in \CC_b(S)$,
\item[$\mathbf{(iii)}$] $\forall f\in \CC_0(S), \ K f \in \CC_b(S)$.
\end{proposition}
\begin{proof}
Remember that, for a Markov kernel $K$ and $f \in \BB_b(S)$, one always has $Kf \in  \BB_b(S)$, so that the only stake in $\mathbf{(ii)}$ and $\mathbf{(iii)}$ is the continuity of $Kf$. Now, by definition, continuity of the map $x \mapsto K(x,\cdot)$ means that, for any sequence $(x_k)_{k \geq 1}$ such that $\lim_{k \to +\infty} x_k = x$, 
we have the weak convergence $K(x_k, \cdot) \xrightarrow[k \to +\infty]{\mathbf{w}} K(x,\cdot)$. In turn, weak convergence means that, for every $f \in \CC_b(S)$, 
$\lim_{k \to +\infty} Kf(x_k) = Kf(x)$, which reads as the continuity of $Pf$ at $x$. As a consequence,  $\mathbf{(i)}$ and $\mathbf{(ii)}$ are equivalent. The equivalence of $\mathbf{(i)}$ and $\mathbf{(iii)}$ stems from the fact that, since $S$ is a lcsm space, the weak convergence  $K(x_k, \cdot) \xrightarrow[k \to +\infty]{\mathbf{w}} K(x,\cdot)$ is equivalent to the fact that $\lim_{k \to +\infty} K f(x_k) = K f(x)$ for every $f \in \CC_0(S)$.
\end{proof}

\begin{proposition}\label{p:continue-0}
A Markov kernel $K$ on a lcsm space $S$ satisfies 
\begin{equation}\label{e:C0C0} \forall f\in \CC_0(S), \ K f \in \CC_0(S),\end{equation}
 if and only if the following two conditions are satisfied:
\begin{itemize}
\item[$\mathbf{(i)}$] $\check{K}$ is a continuous map,
\item[$\mathbf{(ii)}$] for every compact subset $C$ of $S$, $\lim_{x \to \infty} K(x,C)=0$. 
\end{itemize}
\end{proposition}

\begin{proof}
We start with the "if" part, assuming that $\mathbf{(i)}$ and $\mathbf{(ii)}$ are satisfied. Consider $f \in \CC_0(S) $. Since $\CC_0(S) \subset \CC_b(S)$, Proposition \ref{p:continue} and $\mathbf{(i)}$ show that $Kf \in \CC_b(S)$, and it remains to prove that $Kf$ goes to zero at infinity. Given $\varepsilon > 0$ and a compact $C$ such that $|f| \leq \varepsilon$ outside $C$, we have $|Kf(x)| \leq   \n{f}{\infty} K(x,C) + \varepsilon$, so that, thanks to $\mathbf{(ii)}$, $|Kf(x)| \leq 2\varepsilon$ as soon as $x$ is outside a sufficiently large compact set $C'$. 

Now for the "only if" part, assuming that \eqref{e:C0C0} is satisfied. From Proposition \ref{p:continue}, we have $\mathbf{(i)}$, since $\CC_0(S) \subset \CC_b(S)$. As for $\mathbf{(ii)}$, consider a compact set $C$.
By compactness of $C$ and local compactness of $S$, there exists\footnote{This is a rather standard result. Here is an elementary argument, using the notations of the proof of Proposition \ref{p:fortement-continu}.
Since $S$ is locally compact, for every $x \in C$, there exists a number $r > 0$ such that $\overline{B(x,r)}$ is compact. By compactness of $C$, we deduce that there exist $m \geq 1$, $x_1,\ldots,x_m \in K$ and $r_1,\ldots, r_m > 0$ such that $C \subset U$, where $U = \bigcup_{k=1}^m B(x_k,r_k)$. By compactness of $C$ again, the continuous function $x \mapsto d(x, U^c)$ has a minimum value $\rho$ on $C$, and, since $U^c$ is closed and $C \cap U^c = \emptyset$, we must have $\rho > 0$. Choosing $\varepsilon \in (0,\rho)$ ensures that $V_{\varepsilon}(C) \subset U$. As a consequence, $V_{\varepsilon}(C)$ is included in the compact set $\bigcup_{k=1}^m \overline{B(x_k,r_k)}$, so that $\overline{V_{\varepsilon}(C)}$ is compact. One can then take $f = I_C^{\varepsilon}$.} a continuous function $f$ with compact support such that $f \geq \un_C$. Then $K(x,C) = K \un_C (x) \leq Kf(x)$, and, since $f \in \CC_0(S)$, \eqref{e:C0C0} implies that $Kf \in \CC_0(S)$ so that $\lim_{x \to \infty} Kf(x) = 0$.
\end{proof}

We now focus on the case where $S$ is a compact (hence separable) metric space, denoting by $\CC(S)$ the space of continuous real-valued functions on $S$ (due to the fact that $S$ is compact, we have $\CC(S) = \CC_b(S) = \CC_0(S)$). Among the many possible metrics compatible with the weak convergence topology on $\mathcal{P}(S)$, we choose to use the Wasserstein (or Kantorovich-Rubinstein) $W_1$ distance (see \cite{Vil09}, Chapter 6). Specifically, consider the space $\Lip(S)$ of real-valued Lipschitz functions on $S$. Denoting by $d_S$ the metric on $S$, for $f \in \Lip(S)$, we let 
$$\n{f}{\sLip} = \sup \{ |f(y)-f(x)|/d_S(x,y) : \  x , y \in S , \ x \neq y \}.$$
Note that, since $S$ is compact, every $f \in \Lip(S)$ is bounded. 
The $W_1$ distance on $\mathcal{P}(S)$ is then given by
\begin{equation}\label{e:def-beta}W_1(\mu,\nu) = \sup \left\{ \left|  \int_S f(x) d\mu(x) -  \int_S f(x) d\nu(x) \right| :   \    \n{f}{\sLip} \leq 1  \right\}.\end{equation}

\begin{lemma}\label{l:contract-beta}
Let $K,L,M$ be Markov kernels on $S$, then $$\sup\limits_{x \in S} W_1(KL(x,\cdot), KM(x,\cdot)) \leq \sup\limits_{x \in S}  W_1(L(x,\cdot), M(x,\cdot)).$$ 
\end{lemma}
\begin{proof}
 Consider $f \in \Lip(S)$ such that $\n{f}{\sLip} \leq 1$.
By definition, for all $x \in S$, we have that $ \left| L f(x) - M f(x) \right| \leq W_1 \left(L(x,\cdot), M (x,\cdot ) \right),$
so that $\n{Lf - Mf}{\infty} \leq  \sup\limits_{x \in S}  W_1(L(x,\cdot), M(x,\cdot))$. Using the contraction property ($\alpha$) of Markov kernels, $\n{KLf - KMf}{\infty} \leq \n{Lf - Mf}{\infty}$, so that
 $$ \sup\limits_{x \in S}  \sup\limits_{f \in \sLip(S), \ \n{f}{\sLip} \leq 1}\left | KLf(x) - KMf(x)  \right| \leq  \sup\limits_{x \in S}  W_1(L(x,\cdot), M(x,\cdot)).$$
 \end{proof}

\begin{proposition}\label{p:fortement-continu}
Given a family of Markov kernels $(K_t)_{t \in J}$ on a compact metric space $S$, where $J$ is a subset of $(0,+\infty)$ such that $\inf J = 0$, the following properties are equivalent: 
\begin{itemize}
\item[$\mathbf{(i)}$] $\forall f\in \CC(S), \ \lim_{t \to 0+} \n{K_t f - f }{\infty} = 0$
\item[$\mathbf{(ii)}$] for all $\varepsilon > 0$, and all $x \in S$, $K_t(x,B_S(x,\varepsilon)^c)$ goes\footnote{We denote by $B_S(x,r)$ the open ball of radius $r$ centered at $x$.} to $0$ as $t$ goes to $0$, uniformly over $x \in S$, 
\item[$\mathbf{(iii)}$] for all $x \in S$, $W_1(K_t(x,\cdot), \delta_x(\cdot))$ goes to $0$ as $t$ goes to $0$, uniformly over $x \in S$.
\end{itemize}
\end{proposition}

\begin{proof}
Before we start the proof, let us introduce a few notation and definitions. For a non-empty set $A \subset S$ and $x \in S$, we let $d_S(x,A) = \inf \{ d_S(x,y) : y \in A \}$, and, for $r > 0$, we denote $V_r(A) = \{ x \in S : \ d_S(x,A) < r\}$. We then define the function $I_A^r$ on $S$ by $I_A^r(x) = 1 - \frac{\min(d_S(x,A),r)}{r}$. We have that $0 \leq I_A^r \leq 1$, $I_A^r = 1 $ on $A$, $I_A^r = 0$ outside $V_r(A)$, and $\n{I_A^r}{\Lip} \leq 1/r$. 

We start with $\mathbf{(i)} \Rightarrow \mathbf{(ii)}$. Assume that $\mathbf{(i)}$  holds, and consider $\varepsilon > 0$, $x \in S$, and the function $f = I_{B_S(x,\varepsilon/3)}^{\varepsilon/3}$. For all $y \in B_S(x,\varepsilon/3)$ one has that $f(y)=1$, and, for all  $z \notin B_S(y, \varepsilon)$, one has that $z \notin V_{\varepsilon/3}(B_S(x,\varepsilon/3))$ so that $f(z)=0$. Since we also have $f \leq 1$, we deduce that $K_t(y,B_S(y,\varepsilon)^c) \leq 1 - K_tf(y)$.  From $\mathbf{(i)}$, $K_tf(y)$ goes to $f(y)=1$ as $t$ goes to $0$, uniformly with respect to $y \in B_S(x,\varepsilon/3)$, so that $K_t(y,B_S(y,\varepsilon)^c)$ goes to $0$ as $t$ goes to $0$, uniformly with respect to $y \in B_S(x,\varepsilon/3)$. Covering the compact set $S$ by a finite number of balls of the form $B_S(x,\varepsilon/3)$, we have proved  that $\mathbf{(ii)}$ holds. 
 
We now prove that $\mathbf{(ii)} \Rightarrow \mathbf{(i)}$. Assume that $\mathbf{(ii)}$ holds. Consider $f \in \CC(S)$ and $\varepsilon >\nolinebreak 0$. Since $S$ is compact, $f$ is uniformly continuous and there is a $\delta > 0$ such that, whenever $d_S(x,y) \leq \delta$, one has $|f(x)-f(y)| \leq \varepsilon$. As a consequence, for all $x \in S$, 
$$|K_tf(x) - f(x) | \leq  \varepsilon +  2 \n{f}{\infty} K_t(x,B_S(x,\delta)^c ),$$ so $\mathbf{(ii)}$ leads to the desired conclusion.

The proof that  $\mathbf{(ii)} \Rightarrow \mathbf{(iii)}$ is similar, where Lipschitz continuity replaces uniform continuity. Indeed, for $f \in \Lip(S)$, we have, for all $\varepsilon > 0$, and for all $x \in S$, that   
$$|K_tf(x) - f(x) | \leq \n{f}{\sLip} \varepsilon +   \Delta_S \n{f}{\sLip} K_t(x,B_S(x,\varepsilon)^c ),$$ 
using the fact that, for all $x,y \in S$, $|f(y)-f(x)| \leq  \n{f}{\sLip} \Delta_S$, where $\Delta_S = \sup\limits_{x,y \in S} d_S(x,y)$ is finite thanks to the fact that $S$ is compact.

We now prove that $\mathbf{(iii)} \Rightarrow \mathbf{(ii)}$. Given $x \in S$ and $\varepsilon > 0$, consider the function $f =  I_{B_S(x,\varepsilon/3)}^{\varepsilon/3}$ used in the proof that $\mathbf{(i)} \Rightarrow \mathbf{(ii)}$. We have that $\n{f}{\sLip} \leq 3/\varepsilon$, so that $g = (3/\varepsilon)^{-1}f$, sastifies $\n{g}{\sLip} \leq 1$, and as a consequence, for all $y \in S$, $|K_tf(y) - f(y)| \leq (3/\varepsilon) W_1(K_t(y,\cdot), \delta_y(\cdot))$, and we can argue as in the proof that $\mathbf{(i)} \Rightarrow \mathbf{(ii)}$, thanks to $\mathbf{(iii)}$.

\end{proof}

\subsection{Extension of $(P_t)$ to $\overline{\R}$}\label{ss:R-bar}

For $u,v \in \R$, denote by $\CC_{u,v}(\R)$ the set of continuous functions $f:\R \to \R$ such that we have $\lim_{x \to -\infty} f(x) = u$ and $\lim_{x \to +\infty} f(x) = v$.

\begin{proposition}\label{p:Feller-monot}
A Markov kernel $K$ on $\R$ satisfying \eqref{e:C0C0} and 
\begin{equation}\label{e:monot}\forall f\in\BB_b^{\nearrow}(\R) , \ K f \in \BB_b^{\nearrow}(\R), \end{equation}
also satisfies \begin{equation}\label{e:Rbar} \forall f \in \CC_{u,v}(\R) , \ K f \in \CC_{u,v}(\R).\end{equation}
\end{proposition}

\begin{proof}

We first prove that, for all $a > 0$, 
\begin{equation}\label{e:limite-loin} \lim_{x \to -\infty} K(x,[-a,+\infty)) = 0 \mbox{ and } \lim_{x \to +\infty} K(x,(-\infty, a]) = 0.\end{equation}
Given $a > 0$, consider $\varepsilon > 0$. Then choose $b \geq a$ such that $K(0,[-b,+b]^c) \leq \varepsilon$. Thanks to \eqref{e:C0C0} and Proposition \ref{p:continue-0}, there exists  $x_0 \leq 0$ such that, for all $x \leq x_0$, $K(x, [-b,+b]) \leq \varepsilon$. By the monotonicity assumption \eqref{e:monot},   $K(x, (b,+\infty)) \leq K(0, (b,+\infty))) \leq \varepsilon$. Combining the two inequalities for $K(x,\cdot)$, we deduce that $K(x,[-b,+\infty)) \leq 2 \varepsilon$, and, since $b \geq a$, we have that, for all $x \leq x_0$, $K(x, [-a,+\infty)) \leq K(x, [-b,+\infty)) \leq 2 \varepsilon$. We have proved the first part of \eqref{e:limite-loin}. The second part is proved symmetrically.

Now consider $f \in \CC_{u,v}(\R)$. Given $\epsilon > 0$, there exists $a > 0$ such that, for all $z < -a$, $|f(z)-u| \leq \varepsilon$. As a consequence, $|Kf(x) - u| \leq \varepsilon + (||f||_{\infty} + u) K(x,[-a,+\infty))$.
Using the first part of \eqref{e:limite-loin}, we deduce that $\limsup_{x \to -\infty} |Kf(x)-u| \leq \varepsilon$. Since $\varepsilon$ is arbitrary, we have thus proved that $\lim_{x \to -\infty} Kf(x) = u$. A similar argument proves that $\lim_{x \to +\infty} Kf(x) = v$.

Finally, since $f \in \CC_b(\R)$ and $K$ satisfies \eqref{e:C0C0}, Proposition \ref{p:continue-0} shows that $K f$ is a continuous function.

\end{proof}

We equip the extended real line $\overline{\R} = \R \cup \{-\infty,+\infty\} = [-\infty,+\infty]$ with the metric $d_1$ defined by $d_1(x,y) = |\phi(y) - \phi(x)|$, where $\phi(x) = \tanh(x)$ (with $\tanh(+\infty)=1$ and $\tanh(-\infty)=-1$), which makes it a separable compact metric space. For $n \geq 2$, we equip $\overline{\R}^n$ with the metric $d_n$ defined by $d_n(\mathbf{x},\mathbf{y}) = \sum_{i=1}^n |\phi(y_i) - \phi(x_i)|$.

Given a Borel probability measure $\mu$ on $\R$, we extend it to a probability measure $\Tilde\mu$ on $\overline{\R}$ by letting, for every  Borel set $B$ of $\overline{\R}$, $\Tilde{\mu}(B) = \mu(B \cap \R)$.

\begin{lemma}\label{l:cv-bar}
Given Borel probability measures $\mu,\mu_1,\mu_2,\ldots$ on $\R$, $\mu_k \xrightarrow[k \to +\infty]{\mathbf{w}} \mu$ if and only if  $\Tilde{\mu_k} \xrightarrow[k \to +\infty]{\mathbf{w}} \Tilde{\mu}$.
\end{lemma}

\begin{proof}
Assume that $\widetilde{\mu_k} \xrightarrow[k \to +\infty]{\mathbf{w}} \Tilde{\mu}$. For $f \in \CC_0(\R)$, $f$ can be extended to a function $\Tilde{f} \in \CC(\overline{\R})$ by letting $\Tilde{f}(\pm \infty)=0$. With this definition, for any probability measure $\nu$ on $\R$, we have that $ \int_{\R} f(x) d \nu(x) =  \int_{\overline{\R}} \Tilde{f}(x) d \Tilde{\nu}(x)$. The fact that $\lim_{k \to +\infty} \int_{\overline{\R}} \Tilde{f}(x) d \widetilde{\mu_k}(x) =  \int_{\overline{\R}} \Tilde{f}(x) d \Tilde{\mu}(x)$ thus implies that 
 $\lim_{k \to +\infty} \int_{\R} f(x) d \mu_k(x) =  \int_{\R} f(x) d \mu(x)$. So we have proved that $\mu_k \xrightarrow[k \to +\infty]{\mathbf{w}} \mu$. 
 
 Now assume that $\mu_k \xrightarrow[k \to +\infty]{\mathbf{w}} \mu$. For $f \in  \CC(\overline{\R})$, the restriction of $f$ to $\R$ is a bounded continuous function, so we have that $\lim_{k \to +\infty} \int_{\R} f(x) d \mu_k(x) =  \int_{\R} f(x) d \mu(x)$, whence 
 $$\lim_{k \to +\infty} \int_{\overline{\R}} f(x) d \widetilde{\mu_k}(x) =  \int_{\overline{\R}} f(x) d \Tilde{\mu}(x).$$ So we have proved that $\widetilde{\mu_k} \xrightarrow[k \to +\infty]{\mathbf{w}} \Tilde{\mu}$.
\end{proof}

We extend $(P_t)_{t \in \R_+}$ to a family of Markov kernels $(\Tilde{P}_t)_{t \in \R_+}$ on $\overline{\R}$ by setting, for all $x \in \R$: $\Tilde{P}_t(x,\cdot)= \widetilde{P_t(x,\cdot)}$, $\Tilde{P}_t(+\infty,\cdot)=\delta_{+\infty}(\cdot)$ and $\Tilde{P}_t(-\infty,\cdot)=\delta_{-\infty}(\cdot)$. 

\begin{proposition}\label{p:extension-Feller}
If $(P_t)_{t\in\R_+}$ is stochastically monotone and enjoys the Feller property, then $(\Tilde{P}_t)_{t\in\R_+}$ has the Feller property.
\end{proposition}

\begin{proof}
We first prove (Fa), i.e. for every $f\in \CC(\overline{\R})$ and $t \geq 0$, $ \Tilde{P}_t f \in \CC(\overline{\R})$. Given $ f\in \CC(\overline{\R})$, we denote by $f_{|\R}$ the restriction of $f$ to $\R$, and observe that $f_{|\R} \in  \CC_{u,v}(\R)$, where $u=f(-\infty)$ and $v=f(+\infty)$, so, by Proposition \ref{p:Feller-monot}, $ P_t f_{|\R} \in \CC_{u,v}(\R)$. For $x \in \R$, $\Tilde{P}_t f(x) = P_t f_{|\R}(x)$, while $\Tilde{P}_t f(-\infty) = u$ and $\Tilde{P}_t f(+\infty)=v$. We deduce that $\Tilde{P}_t f$ is continuous on $\overline{\R}$.

We now prove (Fb'), i.e. for every $ f\in \CC(\overline{\R})$ and $x \in \overline{\R}$, $\lim_{t \to 0} \Tilde{P}_t f(x) = f(x)$. This is immediate for $x=\pm \infty$ since then $\Tilde{P}_tf(x) = f(x)$ for all $t$. For all $x \in \R$, we have that $P_t(x,\cdot)$ converges weakly to $\delta_x(\cdot)$ as $t \to 0$, so that, since $ f_{|\R} \in \CC_b(\R)$, $\lim_{t \to 0} P_t f_{|\R}(x) = f_{|\R}(x)$, whence $\lim_{t \to 0} \Tilde{P}_t f(x) = f(x)$.

\end{proof}

\begin{remark}
Observe that it is not true, without additional assumptions, that the extension to $\overline{\R}$ of a Feller Markov semigroup on $\R$ always inherits the Feller property (this is true nonetheless for the extension to the one-point compactification $\R \cup \{\infty\}$). Consider for instance the semigroup $(P_t)$ associated with a standard Brownian motion $(B_t)_{t \in \R_+}$ on $\R$ whose sign is reversed at each occurrence of an independent rate $1$ Poisson process $(N_t)_{t \in \R_+}$, i.e. $X_t = B_t \cdot (-1)^{N_t}$. Indeed, for $f \in \CC_{a,b}(\R)$ with $a \neq b$, $\lim_{x \to -\infty} P_tf(x) = a \P(N_t \mbox{ is even})+b \P(N_t \mbox{ is odd})$, so $P_t f \notin \CC_{a,b}(\R)$ for $t > 0$.
\end{remark}

\section{Proof of Theorem \ref{th:extension}}\label{s:extension}

This section is devoted to the proof of Theorem \ref{th:extension}. First, in Section \ref{ss:discrete-flow}, we define the kernels $Q_t^{(n)}$ by using, for $\mathbf{x}=(x_1,\ldots, x_n)$ and $t \geq 0$, the classical comonotone coupling of $P_t(x_1,\cdot),\ldots,P_t(x_n,\cdot)$, which turns out to be order-preserving thanks to the stochastic monotonicity of $(P_t)_{t\in\R_+}$ (more precisely, we work with the extended semigroup $(\Tilde{P}_t)_{t\in\R_+}$ on $\overline{\R}$).
 Then, in Section \ref{ss:construction}, we construct the kernels $P_t^{(n)}$ by iterating the composition of $Q_{s}^{(n)}$ a large number of times for smaller and smaller $s$, then taking the limit. One key tool to prove that this procedure leads to a well-behaved limit is Proposition \ref{p:equicont}, which enables the transfer of regularity properties of the kernels $P_t$ to the kernels $Q_t^{(n)}$, thanks to the interplay between order-preserving properties of kernels and estimates in the  $W_1^{(n)}$ distance.

\subsection{Discrete-time flow}\label{ss:discrete-flow}

For $x \in \R$ and $t \geq 0$, let $F_{x,t}^{[-1]} : [0,1] \to \overline{\R}$ denote the quantile function of the probability distribution $P_t(x,\cdot)$, i.e. $F_{x,t}^{[-1]}(u) = \inf \{ y \in \R :  P_t(x,(-\infty,y] )\geq u \}$.
We extend the definition to $x \in \overline{\R}$ by letting $F_{-\infty,t}^{[-1]} \equiv -\infty$ and $F_{+\infty,t}^{[-1]} \equiv +\infty$.

\begin{lemma}\label{l:pseudo-inv}
The family of functions $F_{x,t}^{[-1]}$ enjoys the following properties:
\begin{itemize}

\item[(a)] for all $x \in  \overline{\R}$ and $t \geq 0$, the map $F_{x,t}^{[-1]}$ is non-decreasing,

\item[(b)] for all $x \in \R$, $t \geq 0$ and $u \in (0,1)$, $F_{x,t}^{[-1]}(u) \in \R$,

\item[(c)] for all $t \geq 0$ and $u \in (0,1)$, the map $x \mapsto F_{x,t}^{[-1]}(u)$ is non-decreasing,

\item[(d)] If $U$ is a random variable following the uniform distribution on the interval $[0,1]$, then, for all $x,t$, the distribution of the random variable $F_{x,t}^{[-1]}(U)$ is $\Tilde{P}_t(x,\cdot)$,

\item[(e)] for all $t \in \R$, and $x \in \overline{\R}$, except for an at most countable set of values of $u \in [0,1]$, the map $y \mapsto  F_{y,t}^{[-1]}(u)$ is continuous at $x$.

\end{itemize}

\end{lemma}

\begin{proof}
Properties (a) and (b) are immediate consequences of the definition and of the fact that the c.d.f. of a real-valued random variable is non-decreasing, with limits $0$ and $1$ at $-\infty$ and $+\infty$, respectively. Property (c) is a consequence of stochastic monotonicity: if $x_1,x_2 \in \R$ are such that $x_1 \leq x_2$, then, for all $y \in \R$, $P_t(x_1,(-\infty,y]) \geq P_t(x_2,(-\infty,y])$. Extension to $\overline{\R}$ is immediate in view of the definition of $F_{x,t}^{[-1]}$ when $x = \pm \infty$. Property (d) is obvious when $x = \pm \infty$, and classical when $x \in \R$, we e.g. refer to the proof of Theorem 25.6 in \cite{Bil95}. We now deal with property (e). Given $t \geq 0$, $x \in \R$ and a sequence $(x_k)_{k \geq 1}$ of real numbers such that $\lim_{k \to +\infty} x_k = x$, we have the weak convergence $P_t(x_k,\cdot)  \xrightarrow[k \to +\infty]{\mathbf{w}}  P_t(x,\cdot)$, so that, refering again to  the proof of Theorem 25.6 in \cite{Bil95}, whenever $F_{x,t}^{[-1]}$ is continuous at a certain $u \in (0,1)$, we have that $\lim_{k \to +\infty} F_{x_k,t}^{[-1]}(u) =  F_{x,t}^{[-1]}(u)$. Since $F_{x,t}^{[-1]}$ is non-decreasing, the number of points of discontinuity is at most countable. Assume now that $x=+\infty$. From the proof of Proposition \ref{p:Feller-monot}, we have that, for any $a > 0$,
$\lim_{y \to +\infty}  P_t(y,(-\infty,a]) = 0$, so that, for any $u \in (0,1)$,   $\lim_{y \to +\infty} F_{y,t}^{[-1]}(u) = +\infty =  F_{+\infty,t}^{[-1]}(u)$. The case $x=-\infty$ is treated symmetrically.

\end{proof}

Given $t \geq 0$ and $\mathbf{x}=(x_1,\ldots, x_n) \in \overline{\R}^n$, define $Q_t^{(n)}(\mathbf{x},\cdot)$ as the distribution on $\overline{\R}^n$ of the random vector $Z_{\mathbf{x}}=\left(F_{x_i,t}^{[-1]}(U) \right)_{1 \leq i \leq n}$, where $U$ is a random variable following the uniform distribution on the interval $[0,1]$. 

\begin{proposition}\label{p:flot-discret}
For all $t \geq 0$, and $n \geq 2$, $Q_t^{(n)}$ is an order-preserving\footnote{We extend the definition of an order-preserving Markov kernel on $\R^n$ to that of an order-preserving Markov kernel on $\overline{\R}^n$ in the obvious way.} Markov kernel on $\overline{\R}^n$. Moreover, for all $t \geq 0$, the family $(Q_t^{(n)})_{n \geq 2}$ is a consistent extension of $\Tilde{P}_t$. 
\end{proposition}

\begin{proof}
By Lemma \ref{l:pseudo-inv} (c), we have that, for any $u \in (0,1)$, $\left(F_{x_i,t}^{[-1]}(u) \right)_{1 \leq i \leq n}$ is order-compatible with $x_1,\ldots,x_n$.
To check that $Q_t^{(n)}$ is a Markov kernel, we note that, by definition, for all $\mathbf{x}$, $Q_t^{(n)}(\mathbf{x},\cdot)$ is a probability measure on $\overline{\R}^n$. Now consider a function $f \in \CC(\overline{\R}^n)$, and define  $Q_t^{(n)} f (\mathbf{x}) = \int_{\overline{\R}^n} f(\mathbf{y}) d Q_t^{(n)}(\mathbf{x},y)$. Given $\mathbf{x} \in \overline{\R}^n$ and a sequence $(\mathbf{x}_k)_{k \geq 1}$ such that 
$\lim_{k \to +\infty} \mathbf{x}_k = \mathbf{x}$, we deduce from Lemma \ref{l:pseudo-inv} (e) that, almost surely, $\lim_{k \to +\infty} Z_{\mathbf{x}_k} = Z_{\mathbf{x}}$, so that, by continuity of $f$, almost surely, $\lim_{k \to +\infty} f(Z_{\mathbf{x}_k}) = f(Z_{\mathbf{x}})$. We thus obtain by dominated convergence that $\lim_{k \to +\infty} Q_t^{(n)} f (\mathbf{x}_k) = Q_t^{(n)} f(\mathbf{x})$. We deduce that $Q_t ^{(n)}f \in  \CC(\overline{\R}^n)$, then invoke Lemma \ref{l:noyau-cont} to deduce that $Q_t^{(n)}$ is a Markov kernel.

That $Q_t^{(n)}$ is an extension of $\Tilde{P}_t$ is a direct consequence of Lemma \ref{l:pseudo-inv} (d). Consistency is a consequence of the fact that, by definition, 
$\pi^n_{i_1,\ldots,i_k}(Z_{\mathbf{x}}) = \left(F_{x_{i_j},t}^{[-1]}(U) \right)_{1 \leq j \leq k} = Z_{\pi^n_{i_1,\ldots,i_k}(\mathbf{x})}$.

\end{proof}

\subsection{Construction of the limit $R^{(n)}_t$}\label{ss:construction}

Now we call $D_+$ the set of positive dyadic rational numbers. For any fixed $t \in D_+$, we write $t=k2^{-m_0}\in D_+$ where $k \geq 1$ and $m_0 \geq 0$ are integers, and $m_0$ has the minimum possible value in such an expression. Then, for every integer $m \geq m_0$, we let
$$Q^{(n),m}_t = \left[Q^{(n)}_{2^{-m}}\right]^{k2^{m-m_0}},$$
i.e. $Q^{(n),m}_t$  is the repeated composition of kernels $Q^{(n)}_{2^{-m}} \cdots Q^{(n)}_{2^{-m}}$ with a total of $k2^{m-m_0}$ kernels in the composition.

As a result, $(Q^{(n),m}_t)_{m \geq m_0}$ is a sequence of Markov kernels on $\overline{\R}^n$. Moreover, from Proposition \ref{p:flot-discret}, for all $t \in D_+$, $n \geq 2$, and $m \geq m_0$, $Q_t^{(n),m}$ is order-preserving\footnote{Using the fact that the composition of order-preserving Markov kernels is still an order-preserving Markov kernel, see Lemma \ref{l:compose-order} in Section \ref{s:appendix-easy}.}, and, for all $t \in D_+$ and $m \geq m_0$, the family $(Q_t^{(n),m})_{n \geq 2}$ is a consistent extension\footnote{Using the fact that, if $(K^{(n)})_{n \geq 2}$ and $(L^{(n)})_{n \geq 2}$ are consistent extensions, respectively of $K$ and $L$, then $(K^{(n)} L^{(n)})_{n \geq 2}$ is a consistent extension of $KL$, see Lemma \ref{l:compose-consistent-ext} in Section \ref{s:appendix-easy}.} of $\Tilde{P}_t$. Now let $\mathbf{Q}_t^{(n),m} =  \check{Q_t^{(n),m}}$.

Denote by $W_1^{(n)}$ the Wasserstein (or Kantorovich-Rubinstein) distance on $\mathcal{P}(\overline{\R}^n)$, where $\overline{\R}^n$ is equipped with the distance $d_n$ defined in Section \ref{ss:R-bar}.
\begin{proposition}\label{p:equicont}
For all $m \geq m_0$, the map $\mathbf{Q}_t^{(n),m}$ from $\overline{\R}^n$ to  $\mathcal{P}(\overline{\R}^n)$ satisfies:
$$W_1^{(n)}\left( \mathbf{Q}_t^{(n),m}(\mathbf{x}), \mathbf{Q}_t^{(n),m}(\mathbf{y})   \right) \leq \sum_{i=1}^n  | \Tilde{P}_t \phi(y_i) - \Tilde{P}_t \phi(x_i)|.$$

\end{proposition}

\begin{proof}
Let $\mathbf{z}=(\mathbf{x},\mathbf{y})=(x_1,\ldots, x_n,y_1,\ldots, y_n)$. Let $(X_1,\ldots, X_n,Y_1,\ldots,Y_n)$ be a random vector on $\overline{\R}^n$ whose distribution is $Q_t^{(2n),m}(\mathbf{z},\cdot)$.
By the consistency property, the distribution of $(X_1,\ldots, X_n)$ is $Q_t^{(n),m}(\mathbf{x},\cdot)$ and the distribution of $(Y_1,\ldots, Y_n)$ is $Q_t^{(n),m}(\mathbf{y},\cdot)$. Moreover, the order-preserving property implies that, with probability $1$, for all $1 \leq i \leq n$, the relative order of $X_i$ and $Y_i$ is the same\footnote{To be explicit: if $x_i \leq y_i$, then almost surely $X_i \leq Y_i$, while, if $y_i \leq x_i$, then almost surely $Y_i \leq X_i$.} as the relative order between $x_i$ and $y_i$. Now consider a function $f \in \Lip(\overline{\R}^n)$ such that $\n{f}{\sLip} \leq 1$. Let $\mathbf{X} = (X_1,\ldots, X_n)$ and $\mathbf{Y} = (Y_1,\ldots, Y_n)$. We have that 
$$\left| Q_t^{(n),m}f(\mathbf{y}) - Q_t^{(n),m}f(\mathbf{x})\right| = \left|  \E f(\mathbf{Y}) -  \E f(\mathbf{X})\right| \leq \E \left| f(\mathbf{Y}) - f(\mathbf{X})\right|.$$ 
Since $\n{f}{\sLip} \leq 1$,
 $$\E \left| f(\mathbf{Y}) - f(\mathbf{X})\right| \leq \E \sum_{i=1}^n | \phi(Y_i) - \phi(X_i) | = \sum_{i=1}^n \E | \phi(Y_i) - \phi(X_i) |.$$
Consider $1 \leq i \leq n$. Since $\phi$ is non-decreasing and the relative order of $X_i$ and $Y_i$ is non-random, we have that 
$$ \E | \phi(Y_i) - \phi(X_i) | = \left|     \E    \left(   \phi(Y_i) - \phi(X_i) \right)       \right| = \left|     \E  \phi(Y_i) -  \E \phi(X_i) \right|.$$
Since $Q_t^{(n),m}$ is an extension of $\Tilde{P}_t$, we have that $ \E  \phi(X_i) =  \Tilde{P}_t \phi(x_i)$ and $ \E  \phi(Y_i) =  \Tilde{P}_t \phi(y_i)$.
\end{proof}

\begin{corollary}\label{c:equicont}
For all $n \geq 2$ and $t \in D_+$, the family $\left(\mathbf{Q}_t^{(n),m} \right)_{m \geq m_0}$  of maps from $\overline{\R}^n$ to $\mathcal{P}(\overline{\R}^n)$ is equicontinuous.
\end{corollary}

\begin{proof}
Note that $\phi \in \CC(\overline{\R})$, so the Feller property of $\Tilde{P}_t$ implies that  $\Tilde{P}_t \phi$ is continuous.
\end{proof}

\begin{proof}[Proof of Theorem \ref{th:extension}.]
Denote by $\mathfrak{C}_n$ the set of continuous functions from $\overline{\R}^n$ to $\mathcal{P}(\overline{\R}^n)$, equipped with the distance $$d_{\mathfrak{C}_n}(\mathbf{K},\mathbf{L}) = \sup\limits_{\mathbf{x} \in \overline{\R}^n} W_1^{(n)}(\mathbf{K}(\mathbf{x}),\mathbf{L}(\mathbf{x})).$$ Remember that $\overline{\R}^n$ is a compact metric space, so that $\mathcal{P}(\overline{\R}^n)$ is also a compact metric space, hence a complete metric space.  As a consequence, for each $t \in D_+$ and $n \geq 2$, in view of Corollary \ref{c:equicont}, we can invoke the Arzelà-Ascoli theorem (see e.g. \cite{Dix84}) to show the convergence along a subsequence: the sequence of maps $\left(\mathbf{Q}_t^{(n),m_k} \right)_{k \geq 1}$ converges, as $k \to +\infty$, to a limiting map $\mathbf{R}_t^{(n)}$, in the sense of uniform convergence of continuous maps from $\overline{\R}^n$ to $\mathcal{P}(\overline{\R}^n)$. Moreover, by diagonal extraction, we can assume that convergence occurs simultaneously for every $t$ in (the countable set) $D_+$ and every $n \geq 2$. To sum up, we have a strictly increasing sequence of integers $(m_k)_{k \geq 1}$ and, for every $t \in D_+$ and $n \geq 2$, an element  $\mathbf{R}_t^{(n)}$ of $\mathfrak{C}_n$, such that\footnote{Note that $\mathbf{Q}_t^{(n),m_k}$ is defined as soon as $m_k \geq m_0$, where $m_0$ depends on $t$.} 
\begin{equation}\label{e:cv-Cn} \lim_{k \to +\infty} d_{\mathfrak{C}_n}(\mathbf{Q}_t^{(n),m_k}  , \mathbf{R}_t^{(n)}) = 0.\end{equation}
Now let $R_t^{(n)} = \hat{\mathbf{R}_t^{(n)}}$, and note that $R_t^{(n)}$  is a Markov kernel (e.g. since $\mathbf{R}_t^{(n)}$ is a continuous map from $\overline{\R}^n$ to $\mathcal{P}(\overline{\R}^n)$, see Lemma \ref{l:noyau-cont}). Moreover, by continuity of the projection maps $\pi^n_{i_1,\ldots,i_k}$ and Lemma \ref{l:order-preserving-limit}, for each $t \in D_+$, the family of Markov kernels $(R_t^{(n)})_{n \geq 2}$ inherits from  $(Q_t^{(n),m_k})_{n \geq 2}$ the property of being a family of order-preserving kernels forming a consistent extension of $\Tilde{P}_t$.

Let us now check that the semigroup property (restricted to $D_+$)  holds for our family of Markov kernels: 
\begin{equation}\label{e:semigroupe-D} \forall s,t \in D_+,  \  R_{s+t}^{(n)} = R_s^{(n)} R_t^{(n)}.\end{equation} By construction, given $s,t \in D_+$, we have that $ Q_s^{(n),m_k}  Q_t^{(n),m_k} =  Q_{s+t}^{(n),m_k}$ for all large enough $k$. Now consider $f \in \Lip(\overline{\R}^n)$ such that $\n{f}{\sLip} \leq 1$, and write 
$$Q_s^{(n),m_k}  Q_t^{(n),m_k} f  -  R_{s}^{(n)}  R_{t}^{(n)}f\!=\! Q_s^{(n),m_k}  Q_t^{(n),m_k} f  - Q_s^{(n),m_k}  R_t^{(n)} f  + Q_s^{(n),m_k}  R_t^{(n)} f  - R_s^{(n)}  R_t^{(n)} f.$$
Using Lemma \ref{l:contract-beta}, we have that,  for all $\mathbf{x} \in \overline{\R}^n$, 
 \begin{equation}\label{e:moitie-1} \left|Q_s^{(n),m_k}  Q_t^{(n),m_k} f(\mathbf{x}) - Q_s^{(n),m_k}  R_t^{(n)} f(\mathbf{x}) \right| \leq d_{\mathfrak{C}_n}(\mathbf{Q}_t^{(n),m_k}  , \mathbf{R}_t^{(n)}).\end{equation}
On the other hand, $R_t^{(n)} f \in \CC( \overline{\R}^n)$ since $f  \in \CC( \overline{\R}^n)$ and $\mathbf{R}_t^{(n)}$ is a continuous map from $\overline{\R}^n$ to $\mathcal{P}(\overline{\R}^n)$ (see Proposition \ref{p:continue}). Since, for every $\mathbf{x} \in \overline{\R}^n$, we have the weak convergence $ Q_s^{(n),m_k}(\mathbf{x},\cdot)   \xrightarrow[k \to +\infty]{\mathbf{w}} R_s^{(n)}(\mathbf{x},\cdot)$, we deduce that 
\begin{equation}\label{e:moitie-2} \lim_{k \to +\infty} Q_s^{(n),m_k} R_t^{(n)} f (\mathbf{x}) =   R_s^{(n)} R_t^{(n)} f (\mathbf{x}).\end{equation}
Combining \eqref{e:cv-Cn}, \eqref{e:moitie-1} and \eqref{e:moitie-2}, we have that $\lim_{k \to +\infty} Q_s^{(n),m_k}  Q_t^{(n),m_k} f (\mathbf{x}) = R_s^{(n)}  R_t^{(n)} f(\mathbf{x})$, for all $\mathbf{x} \in \overline{\R}^n$. Now the weak convergence  $ Q_{s+t}^{(n),m_k}(\mathbf{x},\cdot)  \xrightarrow[k \to +\infty]{\mathbf{w}} R_{s+t}^{(n)}(\mathbf{x},\cdot)$ implies that we have $\lim_{k \to +\infty} Q_{s+t}^{(n),m_k}  f (\mathbf{x}) = R_{s+t}^{(n)} f(\mathbf{x})$. Using the identity  $ Q_s^{(n),m_k}  Q_t^{(n),m_k} =  Q_{s+t}^{(n),m_k}$, we deduce that 
$\lim_{k \to +\infty} Q_s^{(n),m_k}  Q_t^{(n),m_k} f (\mathbf{x}) = R_{s+t}^{(n)} f(\mathbf{x})$, so that 
$ R_s^{(n)}  R_t^{(n)} f(\mathbf{x}) = R_{s+t}^{(n)} f(\mathbf{x})$, and \eqref{e:semigroupe-D} is established\footnote{Indeed, we have established that, for all $\mathbf{x} \in \overline{\R}^n$, $W_1^{(n)}( R_s^{(n)}  R_t^{(n)}(\mathbf{x},\cdot) ,  R_{s+t}^{(n)}(\mathbf{x},\cdot)) =0$.}.

Now for $\mathbf{x} = (x_1,\ldots, x_n)$ and $\varepsilon > 0$, observe that
$B_{\overline{\R}^n}(\mathbf{x},\varepsilon)^c \subset \bigcup_{i=1}^n B_{\overline{\R}}(x_i,\varepsilon/n)^c$, so that, using the fact that $ R_t^{(n)}$ is an extension of $\Tilde{P}_t$ and the union bound, 
$$R_t^{(n)}(\mathbf{x},B_{\overline{\R}^n}(\mathbf{x},\varepsilon)^c ) \leq \sum_{i=1}^n   \Tilde{P}_t(x_i,B_{\overline{\R}}(x_i,\varepsilon/n)^c ).$$
Thanks to the Feller property of $(\Tilde{P}_t)_{t\in\R_+}$ and to Proposition \ref{p:fortement-continu} (ii), we deduce that $ R_t^{(n)}(\mathbf{x},B_{\overline{\R}^n}(\mathbf{x},\varepsilon)^c)$ goes to $0$ as $t$ goes to $0$, uniformly over $\mathbf{x} \in  \overline{\R}^n$, which, thanks to Proposition \ref{p:fortement-continu} again, shows that 
\begin{equation}\label{e:cv-forte-0}\lim_{t \to 0 \atop t \in D_+}\sup\limits_{\mathbf{x} \in  \overline{\R}^n} W_1^{(n)}(R_t^{(n)}(\mathbf{x},\cdot), \delta_{\mathbf{x}}(\cdot)) = 0.\end{equation}
Using Lemma \ref{l:contract-beta}, we have that
\begin{equation}\label{e:contracte-tout-t}\sup\limits_{\mathbf{x} \in \overline{\R}^n}  W_1^{(n)}(R_{s+t}^{(n)}(\mathbf{x},\cdot), R_{s}^{(n)}(\mathbf{x},\cdot)) \leq \sup\limits_{\mathbf{x} \in \overline{\R}^n} W_1^{(n)}(R_t^{(n)}(\mathbf{x},\cdot), \delta_{\mathbf{x}}(\cdot)).\end{equation}

Combining \eqref{e:cv-forte-0} and \eqref{e:contracte-tout-t} shows that the map $s \mapsto \mathbf{R}_s^{(n)}$ from $D_+$ to $\mathfrak{C}_n$ is uniformly continuous. Since $\mathfrak{C}_n$ is a complete metric space, and $D_+$ is a dense subset of $\R_+$, there is a unique extension (see \cite{Dix84}) to a uniformly continuous map $s \mapsto \mathbf{R}_s^{(n)}$ from $\R_+$ to $\mathfrak{C}_n$. This allows us to extend the definition $R_t^{(n)} = \hat{\mathbf{R}_t^{(n)}}$ to every $t \in \R_+$. 
As before, $R_t^{(n)}$ is a Markov kernel (thanks to Lemma \ref{l:noyau-cont} and the fact that $\mathbf{R}_t^{(n)}$ is a continuous map from $\overline{\R}^n$ to $\mathcal{P}(\overline{\R}^n)$), and, by continuity of the projection maps $\pi^n_{i_1,\ldots,i_k}$, the Feller property of $(\Tilde{P}_t)_{t \in \R_+}$ and Lemma \ref{l:order-preserving-limit}, for each $t \in \R_+$, the family of Markov kernels $(R_t^{(n)})_{n \geq 2}$ inherits the property of being a family of order-preserving kernels forming a consistent extension of $\Tilde{P}_t$, from the same property already established for each $t \in D_+$.
Since $\Tilde{P}_0(x,\cdot) = \delta_x(\cdot)$ for all $x \in \overline{\R}$ (or using \eqref{e:cv-forte-0}), we note that $R_0^{(n)}(\mathbf{x},\cdot) = \delta_{\mathbf{x}}(\cdot)$ for all $\mathbf{x} \in {\overline{\R}}^d$. 

We now check that the semigroup property holds for $(R^{(n)}_t)_{t \in \R_+}$, i.e. 
\begin{equation}\label{e:semigroupe} \forall s,t \in \R_+,  \  R_{s+t}^{(n)} = R_s^{(n)} R_t^{(n)}.\end{equation} 
Let $(s_k)_{k \geq 1}$ and $(t_k)_{k \geq 1}$ be sequences of elements in $D_+$ which converge to $s$ and $t$ respectively. By \eqref{e:semigroupe-D}, we have that $R_{s_k}^{(n)}  R_{t_k}^{(n)} = R_{s_k+t_k}^{(n)}$. Now, for $f \in \Lip(\overline{\R}^n)$ such that $\n{f}{\sLip} \leq 1$, we write 
$$R_{s_k}^{(n)}  R_{t_k}^{(n)} f  -  R_{s}^{(n)}  R_{t}^{(n)}f=R_{s_k}^{(n)}  R_{t_k}^{(n)} f - R_{s_k}^{(n)}  R_{t}^{(n)} f  + R_{s_k}^{(n)}  R_{t}^{(n)} f  - R_{s}^{(n)}  R_{t}^{(n)} f.$$ 
Using Lemma \ref{l:contract-beta}, we have that,  for all $\mathbf{x} \in \overline{\R}^n$, 
 $$\left|R_{s_k}^{(n)}  R_{t_k}^{(n)} f(\mathbf{x}) -  R_{s_k}^{(n)}  R_t^{(n)} f(\mathbf{x}) \right| \leq  d_{\mathfrak{C}_n}( \mathbf{R}_{t_k}^{(n)} ,  \mathbf{R}_{t}^{(n)}  ).$$
 Arguing exactly as in the proof of \eqref{e:semigroupe-D}, we have that $R_{t}^{(n)} f  \in \CC( \overline{\R}^n)$. Moreover, by continuity of the extension to $\R_+$, we have the weak convergence $ R_{s_k}^{(n)}(\mathbf{x},\cdot)   \xrightarrow[k \to +\infty]{\mathbf{w}} R_s^{(n)}(\mathbf{x},\cdot)$, so that 
$$\lim_{k \to +\infty} R_{s_k}^{(n)} R_t^{(n)} f (\mathbf{x}) =   R_s^{(n)} R_t^{(n)} f (\mathbf{x}),$$
and similarly 
$$\lim_{k \to +\infty} R_{s_k+t_k}^{(n)} f (\mathbf{x}) =   R_{s+t}^{(n)} f (\mathbf{x}),$$
so that $R_{s+t}^{(n)}f(\mathbf{x}) = R_s^{(n)} R_t^{(n)} f(\mathbf{x})$, and \eqref{e:semigroupe} is established.

Finally, thanks to the continuity of $t \mapsto \mathbf{R}_t^{(n)}$ at $t=0$, and by Proposition \ref{p:fortement-continu}, the family of Markov kernels $(R_t^{(n)})_{n \geq 2}$ satisfies property (Fb).

It remains to define the kernels $P_t^{(n)}$ by restricting the previous construction to $\R^n$. For $n \geq 2$, $t \in \R_+$, $\mathbf{x} \in \R^n$ and Borel set $B$ of $\R^n$, we define 
\begin{equation}
\label{e:def-ptn}
P_t^{(n)}(\mathbf{x},B)=R_t^{(n)}(\mathbf{x},B).
\end{equation}
Since $R_t^{(n)}$ is an extension of $\Tilde{P}_t$ and $\Tilde{P}_t(x_i,\R)=1$, we have that $R_t^{(n)}(\mathbf{x},\R^n)=1$, so $P_t^{(n)}$ is indeed a Markov kernel on $\R^n$, and the fact that, for all $t$, $(P_t^{(n)})_{n \geq 2}$ is a consistent extension of $P_t$ by order-preserving Markov kernels is an immediate consequence of the fact that  $(R_t^{(n)})_{n \geq 2}$ is a consistent extension of $\Tilde{P}_t$ by order-preserving Markov kernels. The semigroup property of $(P_t^{(n)})_{t \in \R_+}$ is also an immediate consequence of the same property for $(R_t^{(n)})_{t \in \R_+}$. We now check the Feller property of $(P_t^{(n)})_{t \in \R_+}$. Since a function $f \in \CC_0(\R^n)$ immediately extends to a function $\Tilde{f} \in \CC(\overline{\R}^n)$, we deduce the fact that $P_t^{(n)} f \in \CC_b(\R^n)$ and the convergence $\lim_{t \to 0}\n{P_t^{(n)}f - f}{\infty} = 0$, from the corresponding properties of $R_t^{(n)}$. To complete the proof of the Feller property of $(P_t^{(n)})_{t \in \R_+}$, it remains to prove that $P_t^{(n)} f \in \CC_0(\R^n)$, i.e. $\lim_{\mathbf{x} \to \infty} P_t^{(n)} f(\mathbf{x})=0$, since we already know that $P_t^{(n)} f \in \CC_b(\R^n)$. To this end, using the fact that $P_t^{(n)}$ is an extension of $P_t$, we have that, for all $a > 0$: 
$$P_t^{(n)}(\mathbf{x},[-a,+a]^n) \leq \min_{1 \leq i \leq n} P_t(x_i, [-a,+a]).$$
Thanks to Proposition \ref{p:continue-0} and the Feller property of $P_t$, we deduce that 
$$\lim_{\mathbf{x} \to \infty} P_t^{(n)}(\mathbf{x},[-a,+a]^n) = 0,$$ 
and, thanks to Proposition \ref{p:continue-0} again, that  $\lim_{\mathbf{x} \to \infty} P_t^{(n)} f(\mathbf{x})=0$.
\end{proof}

\section{Proof of Theorem \ref{th:caract-sm}}\label{s:caract-sm}

This (short) section is devoted to the proof of Theorem \ref{th:caract-sm}. We start with a general result showing that the $\sm$ order is preserved by consistent families of order-preserving Markov kernels. Combined with the maximal property of the comonotone coupling with respect to $\sm$ and the fact that $(\breve{P}_t^{(n)})_{t \in \R_+}$ is obtained as a limit of the composition of order-preserving kernels constructed using the comonotone coupling, we deduce Theorem \ref{th:caract-sm}.

\begin{proposition}\label{p:preserve-sm}
If $(K^{(n)})_{n \geq 1}$ is a consistent family of order-preserving Markov kernels on the successive powers of $\R$, then, for all $n \geq 2$ and $f \in \SM_b(\R^n)$, $K^{(n)} f \in \SM_b(\R^n)$.
\end{proposition}

\begin{proof}
Let $\mathbf{x}=(x_1,\ldots,x_n)$, $\mathbf{y}=(y_1,\ldots,y_n)$ and $\mathbf{z}=(\mathbf{x},\mathbf{y})=(x_1,\ldots, x_n,y_1,\ldots, y_n)$. Let $(X_1,\ldots, X_n,Y_1,\ldots,Y_n)$ be a random vector on $\overline{\R}^n$ whose distribution is $K^{(2n)}(\mathbf{z},\cdot)$, and let $\mathbf{X}=(X_1,\ldots,X_n)$ and $\mathbf{Y}=(Y_1,\ldots,Y_n)$. By the consistency property,  the distribution of $\mathbf{X}$ is $K^{(n)}(\mathbf{x},\cdot)$ and  the distribution of $\mathbf{Y}$ is $K^{(n)}(\mathbf{y},\cdot)$, so that $ K^{(n)}f(\mathbf{x}) = \E f(\mathbf{X})$ and $ K^{(n)}f(\mathbf{y}) = \E f(\mathbf{Y})$. Now, for $1 \leq i \leq n$, let $Z_i^+ = Y_i$ and $Z_i^- = X_i$ when $x_i \leq y_i$, $Z_i^+ = X_i$ and $Z_i^- = Y_i$ when $x_i > y_i$, so that, thanks to the consistency property again, the distribution of $(Z_1^+,\ldots, Z_n^+)$ is $K^{(n)}(\mathbf{x} \vee \mathbf{y},\cdot)$ and  the distribution of $(Z_1^-,\ldots, Z_n^-)$ is $K^{(n)}(\mathbf{x} \wedge \mathbf{y},\cdot)$, so that $ K^{(n)}f(\mathbf{x} \vee \mathbf{y}) = \E f(Z_1^+,\ldots, Z_n^+)$ and $ K^{(n)}f(\mathbf{x} \wedge \mathbf{y}) = \E f(Z_1^-,\ldots, Z_n^-)$.
Moreover, by the  order-preserving property of $K^{(2n)}$, we have that $\mathbf{X} \vee \mathbf{Y} = (Z_1^+,\ldots, Z_n^+)$ and $\mathbf{X} \wedge \mathbf{Y} =  (Z_1^-,\ldots, Z_n^-)$. By the super-modularity property of $f$, 
$f(\mathbf{X} \vee \mathbf{Y})+ f(\mathbf{X} \wedge \mathbf{Y}) \geq f(\mathbf{X}) + f(\mathbf{Y})$, and we thus deduce that 
$ f(Z_1^+,\ldots, Z_n^+) +  f(Z_1^-,\ldots, Z_n^-) \geq f(\mathbf{X}) + f(\mathbf{Y})$. Taking expectations, we deduce that 
$K^{(n)}f(\mathbf{x} \vee \mathbf{y})+K^{(n)}f(\mathbf{x} \wedge \mathbf{y}) \geq K^{(n)}f(\mathbf{x}) + K^{(n)}f(\mathbf{y})$, so we have proved that $K^{(n)}f$ is super-modular.
\end{proof}

Remember the family of Markov kernels $(Q_t^{(n)})_{t\in\R_+}$ introduced in Section \ref{ss:discrete-flow}. These are defined as Markov kernels on $\bar\R^n$, but, for $\mathbf x\in\R^n$, $Q^{(n)}_t(\mathbf x, \R^n)=1$, so, letting $T_t^{(n)}(\mathbf x, B)=Q^{(n)}_t(\mathbf x, B)$ for every Borel set $B$ of $\R^n$, we obtain a family of Markov kernels $(T_t^{(n)})_{t\in\R_+}$ on $\R^n$. 

\begin{proposition}\label{p:comon-step}
 Let $(M_t^{(n)})_{t\in\R_+} \in \mathfrak{M}_{n}$. Then, for all $t \geq 0$ and $\mathbf{x} \in \R^n$, we have that $M_t^{(n)}(\mathbf{x},\cdot) \sm T_t^{(n)}(\mathbf{x},\cdot)$. 
\end{proposition}

\begin{proof}
Let $\mathbf{x}=(x_1,\ldots,x_n)$. By construction (see Theorem 3.1.1 in \cite{MulSto02}), $T_t^{(n)}(\mathbf{x},\cdot)$ is a comonotone probability distribution on $\R^n$. Moreover, $T_t^{(n)}(\mathbf{x},\cdot)$ and $M_t^{(n)}(\mathbf{x},\cdot)$ have the same one-dimensional marginals, namely the distributions $P_t(x_1,\cdot),\ldots,P_t(x_n,\cdot)$. By Theorem 3.9.8 in \cite{MulSto02} (using property (P5)), we deduce that $M_t^{(n)}(\mathbf{x},\cdot) \sm T_t^{(n)}(\mathbf{x},\cdot)$.
\end{proof}

\begin{proof}[Proof of Theorem \ref{th:caract-sm}]
\nopagebreak
We now prove by induction that, for all $f \in \SM_b(\R^n)$, $s \geq 0$, and integer $k \geq 0$, \begin{equation}\label{e:compare-progressive}M_{sk}^{(n)}f \leq \left[T_s^{(n)} \right]^k f.\end{equation} For $k=0$, this is obvious since the property to be proved is that $f \leq f$.
Now write $\left[T_s^{(n)} \right]^{k+1} f = T_s^{(n)} \left[T_s^{(n)} \right]^k f$. Thanks to Proposition \ref{p:preserve-sm}, $\left[T_s^{(n)} \right]^k f$ is super-modular. We deduce from Proposition \ref{p:comon-step} that 
$M_s^{(n)} \left[T_s^{(n)} \right]^k f \leq T_s^{(n)} \left[T_s^{(n)} \right]^k f =  \left[T_s^{(n)} \right]^{k+1}f$. Assuming that  $M_{sk}^{(n)}f \leq \left[T_s^{(n)} \right]^k f$, positivity implies that $M_{s(k+1)}^{(n)}f = M_s^{(n)}  M_{sk}^{(n)}f \leq  M_s^{(n)} \left[T_s^{(n)} \right]^k f $,  so that, combining the previous inequalities, we deduce that $M_{s(k+1)}^{(n)}f \leq \left[T_s^{(n)} \right]^{k+1}f$, and \eqref{e:compare-progressive} is proved by induction. 

Assume that $f$ is continuous in addition to being in $\SM_b(\R^n)$.
Given $t \in D_+$, remember from Section \ref{ss:construction} that, for $m \geq m_0$, $Q^{(n),m}_t = \left[Q^{(n)}_{2^{-m}}\right]^{k2^{m-m_0}},$ where $t=k2^{-m_0}\in D_+$, and $k \geq 1$ and $m_0 \geq 0$ are integers.
As a consequence, if we write $T^{(n),m}_t=\left[T^{(n)}_{2^{-m}}\right]^{k2^{m-m_0}}$, then \eqref{e:compare-progressive} implies that 
\begin{equation}
\label{e:compare-progressive-dyadic}
M_t^{(n)}f \leq T^{(n),m}_t f
\end{equation} 
Remember that $R_t^{(n)}$ is defined in Section \ref{ss:construction} as the limit for $k \to +\infty$ of a subsequence $\left(Q^{(n),m_k}_t\right)_{k \geq 1}$. For all $\mathbf x\in\R^n$, we have $Q^{(n),m_k}_t(\mathbf x,\R^n)=1$ for every $k \geq 1$, and $R_t^{(n)}(\mathbf x,\R^n)=1$. Since $T^{(n),m_k}_t$ and $\breve{P}^{(n)}_t$ are respectively defined by restricting $Q^{(n),m_k}_t$ and $R_t^{(n)}$ to $\R^n$, we have that\footnote{The argument is the same as in the proof of Lemma \ref{l:cv-bar}.} $\breve{P}^{(n)}_t(\mathbf x,\cdot)$ is the weak limit of the sequence $\left(T^{(n),m_k}_t(\mathbf x,\cdot)\right)_{k\geq 1}$. Taking the limit in \eqref{e:compare-progressive-dyadic}, we thus have that $M^{(n)}_tf\leq \breve{P}_t^{(n)}f$. Given $t \in \R_+$, consider a sequence  $(s_m)_{m \geq 1}$ of elements of $D_+$ such that $\lim_{m \to +\infty} s_m = t$.
 By the Feller property of $(M_s^{(n)})_{s \in \R_+}$ and $(\breve{P}_s^{(n)})_{s \in \R_+}$, we have that, as $m \to +\infty$, $M_{s_m}^{(n)}f$ converges pointwise to $M_t^{(n)} f$, and  $\breve{P}_{s_m}^{(n)}f$ converges pointwise to $\breve{P}_t^{(n)} f$, so that, 
taking the limit in the inequality $M_{s_m}^{(n)}f \leq \breve{P}_{s_m}^{(n)} f$, we have that $M_t^{(n)}f \leq \breve{P}_t^{(n)} f$. Since we know that, for all $\mathbf{x}=(x_1,\ldots,x_n) \in \R^n$,  $M_t^{(n)}(\mathbf{x},\cdot)$ and $\breve{P}_t^{(n)}(\mathbf{x},\cdot)$ have the same one-dimensional marginals (namely $P_t(x_1,\cdot),\ldots,P_t(x_n,\cdot)$), we can invoke Theorems 3.9.10 and 3.9.11 in \cite{MulSto02} to deduce the conclusion of Theorem \ref{th:caract-sm}.
\end{proof}

\begin{remark}
Using the same arguments than in the proof of Theorem \ref{th:caract-sm}, we can prove that, for all $\mathbf{x} \in \R^n$ and $t \in D_+$, one has the comparison $\left[T^{(n)}_{t/2}\right]^2(\mathbf x,\cdot)\sm T^{(n)}_{t}(\mathbf x, \cdot)$. In turn, this implies that the sequence of probability measures $\left(T^{(n),m}_t(\mathbf{x},\cdot) \right)_{m \geq m_0}$ is non-increasing with respect to $\sm$. On the other hand, we have the weak convergence $T^{(n),m_k}_t(\mathbf{x},\cdot) \xrightarrow[k \to +\infty]{\mathbf{w}} \breve{P}_t^{(n)}(\mathbf{x},\cdot)$. Thus, for any continuous function $f \in \SM_b(\R^n)$, we have that at the same time that the sequence $\left(T^{(n),m}_t f(\mathbf{x})\right)_{m \geq m_0}$ is non-increasing and that $\lim_{k \to +\infty} T^{(n),m_k}_t f(\mathbf{x}) = \breve{P}_t^{(n)}f(\mathbf{x})$, so we have the convergence $\lim_{m \to +\infty} T^{(n),m}_t f(\mathbf{x}) = \breve{P}_t^{(n)}f(\mathbf{x})$, without the need to consider convergence along a subsequence. We deduce\footnote{E.g. using the fact that the sequence $\left(T^{(n),m}_t(\mathbf{x},\cdot) \right)_{m \geq m_0}$ is tight since its one-dimensional marginals are fixed, and that, from Theorems 3.9.10 and 3.9.11 in \cite{MulSto02}, the bounded continuous super-modular functions separate Borel probability measures on $\R^d$.} that the sequence $\left(T^{(n),m}_t(\mathbf{x},\cdot) \right)_{m \geq m_0}$ converges weakly to $\breve{P}_t^{(n)}(\mathbf{x},\cdot)$. This points to a possible alternative approach to (at least some parts of) Theorem \ref{th:extension}, exploiting monotonicity with respect to the $\sm$ order,  which we have not tried to push further. Such an approach would be reminiscent of the one used in a different context in \cite{BouJui22}.
\end{remark}

\providecommand{\bysame}{\leavevmode\hbox to3em{\hrulefill}\thinspace}
\providecommand{\MR}{\relax\ifhmode\unskip\space\fi MR }
\providecommand{\MRhref}[2]{%
  \href{http://www.ams.org/mathscinet-getitem?mr=#1}{#2}
}
\providecommand{\href}[2]{#2}

The authors would like to thank the ANR SOCOT, Ecole des Mines de Nancy and ENS Rennes for their support.

\begin{appendix}

\section{About Theorems 4 and 5 in \cite{Kam+77}}
\label{s:appendix-KamKreOBr}

In this section, we explain why we think an element is missing in the proof of Theorem 4  in \cite{Kam+77}, and hence of Theorem 5  in \cite{Kam+77}, which is a corollary. 

\subsection{About Theorem 4}

Theorem 4  in \cite{Kam+77} considers a partially ordered Polish space $(E, \preccurlyeq)$, two families of $E$-valued random variables $(X_t)_{t \in \R_+}$ and $(Y_t)_{t \in \R_+}$ such that $t \mapsto X_t$ and $t \mapsto Y_t$ are (random) càdlàg paths from $\R_+$ to $E$, and, assumes that, for all $n \geq 2$,  $x_1,\ldots, x_{n-1} \in \R$, $y_1,\ldots, y_{n-1} \in \R$ such that   $x_i \preccurlyeq y_i$ for $i=1,\ldots,n-1$, and {\bf ordered time indices} $0 \leq t_1 < \cdots < t_n \in \R$, one has \begin{equation}\label{e:stoch-dom-KKO}\P(X_{t_n} \in \cdot |X_{t_1}=x_1,\cdots,X_{t_{n-1}}=x_{n-1}) \st 
\P(Y_{t_n} \in \cdot |Y_{t_1}=y_1,\cdots,Y_{t_{n-1}}=y_{n-1}),\end{equation} where $\st$ denotes stochastic domination between probability measures on $(E, \preccurlyeq)$ (see \cite{Kam+77}).

Assuming moreover that $\P(X_0 \in \cdot) \st \P(Y_0 \in \cdot)$, the conclusion of Theorem 4 is that one can define a family of pairs of $E$-valued random variables $(\tilde{X}_{t},\tilde{Y}_{t})_{t\in\R_+}$ on the same probability space, in such a way that $t \mapsto \tilde{X}_t$ and $t \mapsto \tilde{Y}_t$ are (random) càdlàg paths from $\R_+$ to $E$, that almost surely for all $t$,  $\tilde{X}_{t} \preccurlyeq \tilde{Y}_{t}$, that $(\tilde{X}_{t})_{t\in\R_+}$ has the same (joint) distribution as $(X_{t})_{t\in\R_+}$, and that $(\tilde{Y}_{t})_{t\in\R_+}$ has the same (joint) distribution as $(Y_{t})_{t\in\R_+}$.

We now consider the proof of Theorem 4. Given an {\bf increasing sequence} $0 = t_1 < t_2 < \cdots$, one can use \eqref{e:stoch-dom-KKO} and Strassen's theorem to define inductively a sequence of pairs of $E$-valued random variables $(\tilde{X}_{t_n},\tilde{Y}_{t_n})_{n \geq 1}$ in such a way that, almost surely, for all $n \geq 1$, $\tilde{X}_{t_n} \preccurlyeq \tilde{Y}_{t_n}$, and such that $(\tilde{X}_{t_1}, \ldots, \tilde{X}_{t_n})$ has the same distribution as $(X_{t_1},\ldots, X_{t_n})$, and $(\tilde{Y}_{t_1},\ldots, \tilde{Y}_{t_n})$ has the same distribution as $(Y_{t_1},\ldots, Y_{t_n})$. (This mirrors the argument used to prove Theorem 2 in \cite{Kam+77}.) 

The problem is the claim (without an explanation), in the proof of Theorem 4, that this approach can be used to produce a sequence of random variables $(\tilde{X}_{t_n},\tilde{Y}_{t_n})_{n \geq 1}$ with the above properties, where $(t_n)_{n \geq 1}$ is an enumeration of a countable dense subset of $\R_+$, since clearly such an enumeration {\bf cannot be done using an increasing sequence}. Thus, to make the proof of Theorem 4 work, it seems that the following strengthening of Assumption \eqref{e:stoch-dom-KKO} would be needed, allowing for unordered sequences of time indices: 
 for every $n \geq 2$, $x_1,\ldots, x_{n-1} \in \R$, $y_1,\ldots, y_{n-1} \in \R$ such that   $x_i \preccurlyeq y_i$ for $i=1,\ldots,n-1$, and {\bf pairwise distinct time indices} $t_1, \ldots, t_n \in \R_+$, one has 
  \begin{equation}\label{e:stoch-dom-KKO-bis}\P(X_{t_n} \in \cdot |X_{t_1}=x_1,\cdots,X_{t_{n-1}}=x_{n-1}) \st 
\P(Y_{t_n} \in \cdot |Y_{t_1}=y_1,\cdots,Y_{t_{n-1}}=y_{n-1}).\end{equation}
  
 \subsection{About Theorem 5}

Theorem 5  in \cite{Kam+77} assumes that $(X_t)_{t \in \R_+}$ and $(Y_t)_{t \in \R_+}$ are (possibly inhomogeneous in time) Markov processes, and does not directly assume \eqref{e:stoch-dom-KKO}, but instead the following stochastic monotonicity condition: for all $s,t \geq 0$, and all $x,y \in E$ such that $x \preccurlyeq y$,
  \begin{equation}\label{e:stoch-dom-KKO-ter} \P(X_{t+s} \in \cdot |X_{t}=x) \st  \P(Y_{t+s} \in \cdot |Y_{t}=y),\end{equation}
  with the same conclusion as Theorem 4. The proof of Theorem 5 consists in observing that, using the Markov property, \eqref{e:stoch-dom-KKO-ter} implies \eqref{e:stoch-dom-KKO}, then invoking Theorem 4. However, in view of our previous remark, this approach does not seem to lead to a complete proof of Theorem 5 unless one can prove that 
    \eqref{e:stoch-dom-KKO-ter} implies \eqref{e:stoch-dom-KKO-bis}. Unfortunately, such an implication is not true in general when $(t_k)_{1\leq k\leq n}$ is not increasing, as we show in the following counterexample.

   Take $n=3, t_1=0,t_2=2,t_3=1$ and $E = \{a,b,c\}$, where $a \preccurlyeq b \preccurlyeq c$, and where $X$ and $Y$ are two versions of the same continuous-time Markov chain on $E$ with distinct starting points. The infinitesimal generator of the chain is prescribed by the $Q$-matrix (see e.g. \cite{And91}) $$
L = (L_{xy})_{x \in E, y \in E} = \begin{blockarray}{cccc}
 & a & b & c \\
\begin{block}{c(ccc)}
  a & -2.5 & 1.75 & 0.75 \\
  b & 1.5 & -2.5 & 1  \\
  c & 0.5 & 0 & -0.5 \\
\end{block}
\end{blockarray}
 $$
and defines a Markov semigroup on $E$, given, for all $t \geq 0$, by the transition matrix $P_t = \exp(tL)$. To check that $L$ indeed defines a stochastically monotone Markov semigroup on $E$, we check the condition stated in \cite{And91} (Theorem 3.4  page 249, attributed to \cite{Kir76}), which in our setting reduces to the two conditions $L_{ac} \leq L_{bc}$ and $L_{ba} \geq L_{ca}$, both visibly satisfied. We now check numerically that $$\P(X_1 \in \cdot | X_0 = a, X_2 = a) \st\P(Y_1 \in \cdot | Y_0 = b, Y_2 = b).$$ Denoting $P=\exp(L)$, we have that 
 $$\P(X_1 \preccurlyeq a | X_0 = a, X_2 = a) = \P(X_1 = a | X_0 = a, X_2 = a) = \frac{P_{aa} P_{aa}}{P_{aa}P_{aa} + P_{ab}P_{ba} + P_{ac}P_{ca}},$$
 and 
 $$\P(Y_1 \preccurlyeq a | Y_0 = b, Y_2 = b) = \P(Y_1 = a | Y_0 = b, Y_2 = b) = \frac{P_{ba} P_{ab}}{P_{ba}P_{ab} + P_{bb}P_{bb} + P_{bc}P_{cb}}.$$
 Numerically computing $P$ with the SciPy open-source software\footnote{Specifically, we used {\tt scipy.linalg.expm} to compute the matrix exponential.}, we get, rounding to $3$ decimal places, that $\P(X_1 \preccurlyeq a | X_0 = a, X_2 = a)  \approx 0.362$ and $\P(Y_1 \preccurlyeq a | Y_0 = b, Y_2 = b)  \approx 0.374$, while the stochastic domination $\P(X_1 \in \cdot | X_0 = a, X_2 = a) \st \P(Y_1 \in \cdot | Y_0 = b, Y_2 = b)$ would imply the inequality $ \P(Y_1 \preccurlyeq a | Y_0 = b, Y_2 = b) \leq \P(X_1 \preccurlyeq a | X_0 = a, X_2 = a)$.

To sum up, an element seems to be missing in the proof of Theorem 5. We observe that, in the specific case where $E = \R$ and where $(X_t)_{t\in\R_+}$ and $(Y_t)_{t\in\R_+}$ are two versions of the same Feller process with distinct starting points, Theorem \ref{th:extension} in the present paper can be used to deduce the conclusion of Theorem 5.

\section{Miscellaneous lemmas}
\label{s:appendix-easy}

\begin{lemma}\label{l:compose-order}
If $K^{(n)}$ and $L^{(n)}$ are order-preserving kernels on $\R^n$, so is their composition $K^{(n)}L^{(n)}$.
\end{lemma}

\begin{proof}
By definition, for all $\mathbf{x} \in \R^n$, $K^{(n)}L^{(n)}(\mathbf{x}, \R^n_{\mathbf{x}}) = \int_{\R^n} L^{(n)}(\mathbf{y}, \R^n_{\mathbf{x}}) dK^{(n)}(\mathbf{x},\mathbf{y})$.
Since $K^{(n)}$ is order-preserving, we have that  $K^{(n)}(\mathbf{x}, \R^n_{\mathbf{x}})=1$, so the previous integral can be rewritten as  $\int_{\R^n_{\mathbf{x}}} L^{(n)}(\mathbf{y}, \R^n_{\mathbf{x}}) dK^{(n)}(\mathbf{x},\mathbf{y})$.
For all $\mathbf{y} \in \R^n_{\mathbf{x}}$, the definition shows that\footnote{If $\mathbf{x}=(x_1,\ldots,x_n)$,  $\mathbf{y}=(y_1,\ldots,y_n)$, $\mathbf{z}=(z_1,\ldots,z_n)$ are such that $\mathbf{y} \in \R^n_{\mathbf{x}}$ and $\mathbf{z} \in \R^n_{\mathbf{y}}$, we have that, for all $1 \leq i,j \leq n$, $x_i \leq x_j  \Rightarrow y_i \leq y_j$ (since $\mathbf{y} \in \R^n_{\mathbf{x}}$), hence $x_i \leq x_j  \Rightarrow z_i \leq z_j$ (since $\mathbf{z} \in \R^n_{\mathbf{y}}$).} $\R^n_{\mathbf{y}} \subset \R^n_{\mathbf{x}}$, so that, since $L^{(n)}$ is order-preserving, we have $L^{(n)}(\mathbf{y}, \R^n_{\mathbf{x}})=1$, and we have proved that $K^{(n)}L^{(n)}(\mathbf{x}, \R^n_{\mathbf{x}}) =1$.
\end{proof}

\begin{lemma}\label{l:compose-consistent-ext}
If $(K^{(n)})_{n \geq 2}$ and $(L^{(n)})_{n \geq 2}$ are consistent extensions, respectively of $K$ and $L$, then $(K^{(n)} L^{(n)})_{n \geq 2}$ is a consistent extension of $KL$.
\end{lemma}
\begin{proof}
Let $K^{(1)}=K$ and $L^{(1)}=L$. Given integers $1 \leq k \leq n$,  $i_1,\ldots, i_k \in \ent{1}{n}$, $\mathbf{x} \in S^n$ and measurable subset $B$ of $S^k$, we have by definition 
$$K^{(n)} L^{(n)}(\mathbf{x}, (\pi^n_{i_1,\ldots,i_k})^{-1}(B))=  \int_{S^n} L^{(n)}(\mathbf{y}, (\pi^n_{i_1,\ldots,i_k})^{-1}(B)) dK^{(n)}(\mathbf{x},\mathbf{y}).$$
Since $(L^{(n)})_{n \geq 2}$ is a consistent extension of $L$, we have that 
$$ L^{(n)}(\mathbf{y}, (\pi^n_{i_1,\ldots,i_k})^{-1}(B)) =  L^{(k)}(\pi^n_{i_1,\ldots,i_k}(\mathbf{y}), B),$$ 
so that 
\begin{equation}\label{e:integrale-mesure}K^{(n)} L^{(n)}(\mathbf{x}, (\pi^n_{i_1,\ldots,i_k})^{-1}(B))=  \int_{S^n}  L^{(k)}(\pi^n_{i_1,\ldots,i_k}(\mathbf{y}), B) dK^{(n)}(\mathbf{x},\mathbf{y}).\end{equation}
Denoting by $(\pi^n_{i_1,\ldots,i_k})_{*}K^{(n)}(\mathbf{x},\cdot)$ the image probability measure on $S^k$ defined by $$\left[(\pi^n_{i_1,\ldots,i_k})_{*}K^{(n)}(\mathbf{x},\cdot)\right](A) = K^{(n)}(\mathbf{x}, (\pi^n_{i_1,\ldots,i_k})^{-1}(A),$$
\eqref{e:integrale-mesure} rewrites as  
$$K^{(n)} L^{(n)}(\mathbf{x}, (\pi^n_{i_1,\ldots,i_k})^{-1}(B))=  \int_{S^k}  L^{(k)}(\mathbf{z}, B) d\left[(\pi^n_{i_1,\ldots,i_k})_{*}K^{(n)}(\mathbf{x},\cdot)\right](\mathbf{z}).$$
On the other hand, since $(K^{(n)})_{n \geq 2}$ is a consistent extension of $K$, we have that  
 $$\left[(\pi^n_{i_1,\ldots,i_k})_{*}K^{(n)}(\mathbf{x},\cdot)\right](A) = K^{(k)}(\pi^n_{i_1,\ldots,i_k}(\mathbf{x}),A),$$
 so that 
 $$K^{(n)} L^{(n)}(\mathbf{x}, (\pi^n_{i_1,\ldots,i_k})^{-1}(B))=  \int_{S^k}  L^{(k)}(\mathbf{z}, B) d K^{(k)}(\pi^n_{i_1,\ldots,i_k}(\mathbf{x}),\mathbf{z}),$$
and thus 
$$K^{(n)} L^{(n)}(\mathbf{x}, (\pi^n_{i_1,\ldots,i_k})^{-1}(B)) = (K^{(k)}   L^{(k)}) (\pi^n_{i_1,\ldots,i_k}(\mathbf{x}), B),$$
which proves the conclusion of the Lemma.
\end{proof}

\begin{lemma}\label{l:order-preserving-limit}
If $K$ is a Markov kernel on $\overline{\R}^n$ and $K_1,K_2,\ldots$ is a sequence of order-preserving Markov kernels on $\overline{\R}^n$ such that, for all $\mathbf{x} \in \overline{\R}^n$, 
$K_k(\mathbf{x},\cdot)  \xrightarrow[k \to +\infty]{\mathbf{w}} K(\mathbf{x},\cdot)$, 
then $K$ is order-preserving.
\end{lemma}
\begin{proof}
For any $\mathbf{x} \in \overline{\R}^n$, the set $\overline{\R}^n_{\mathbf{x}} = \{ \mathbf{y} \in \overline{\R}^n \mbox{ such that $\mathbf{y}$ is order-compatible with $\mathbf{x}$} \}$ is a closed subset of $\overline{\R}^n$. As a consequence, by the assumed weak convergence, we have that, for all $\mathbf{x} \in  \overline{\R}^n$, 
$\limsup_{k \to +\infty} K_k(\mathbf{x},\overline{\R}^n_{\mathbf{x}}) \leq K(\mathbf{x},\overline{\R}^n_{\mathbf{x}})$. Since $K_k$ is order-preserving for all $k$, we have that $ K_k(\mathbf{x},\overline{\R}^n_{\mathbf{x}})=1$, and we deduce that $K(\mathbf{x},\overline{\R}^n_{\mathbf{x}}) = 1$.
\end{proof}

\section{Numerical illustrations}
\label{s:appendix-simulations}

In this section, we show some numerical simulations to illustrate our construction, when the underlying process is a square-root diffusion as defined by the following s.d.e.\footnote{This class of processes has been studied by Feller \cite{Fel51}, and gained popularity in financial mathematics as a model for the dynamics of interest rates \cite{Cox+85}. }
\begin{equation}\label{e:sde}d X_t = a(b-X_t) dt + \sigma \sqrt{X_t} dW_t,\end{equation}
where $(W_t)_{t\in\R_+}$ is a standard one-dimensional Brownian motion, and where $a>0$, $b>0$, and $\sigma > 0$. This defines an explicit Feller semigroup on $[0,+\infty[$: for $t> 0$, $P_t(x,\cdot)$ is the distribution of  
$Z/(2*c)$, where $c = \frac{2a}{(1-e^{-at})\sigma^2},$ and $Z$ follows a non-central chi-square distribution with $\frac{4ab}{\sigma^2}$ degrees of freedom and non-centrality parameter $2 c x e^{-at}$. 

This example does not fit exactly into the framework of the present paper, which is devoted to Feller processes on the full real line $\R$. However, our construction of the comonotone flow would work as well on $[0,+\infty)$ instead of $\R$, and we stick to this example as it is likely the simplest and best-known example of an explicitly solvable s.d.e. with non-constant diffusion coefficient.

We work with the set of parameters $a=3$, $b=2$ and $\sigma^2 = 8$. The following figures depict the values of $N=5000$ simulated i.i.d. pairs distributed according to $Q^{(2),m}_{t}((x_1,x_2),\cdot)$, for $x_1=0.5$, $x_2=2$, $t=0.5$, and $m=1,\ldots,6$. Figure \ref{fig:comonotone-xy} displays each pair $(x,y)$ as a point with coordinates $(x,y)$, while Figure \ref{fig:comonotone-x-y-x} uses the point with coordinates $(x,y-x)$ instead. Simulations were done using the R open-source software. 

For $m=1$, $Q^{(2),m}_{t}((x_1,x_2),\cdot)$ is just the classical comonotone coupling, which is the reason why the points are lying on a curve. As $m$ gets larger, the array of points gets more spread out, with barely distinguishable differences between the few last consecutive graphs.

\begin{figure}
\centering
 \includegraphics[scale=0.75]{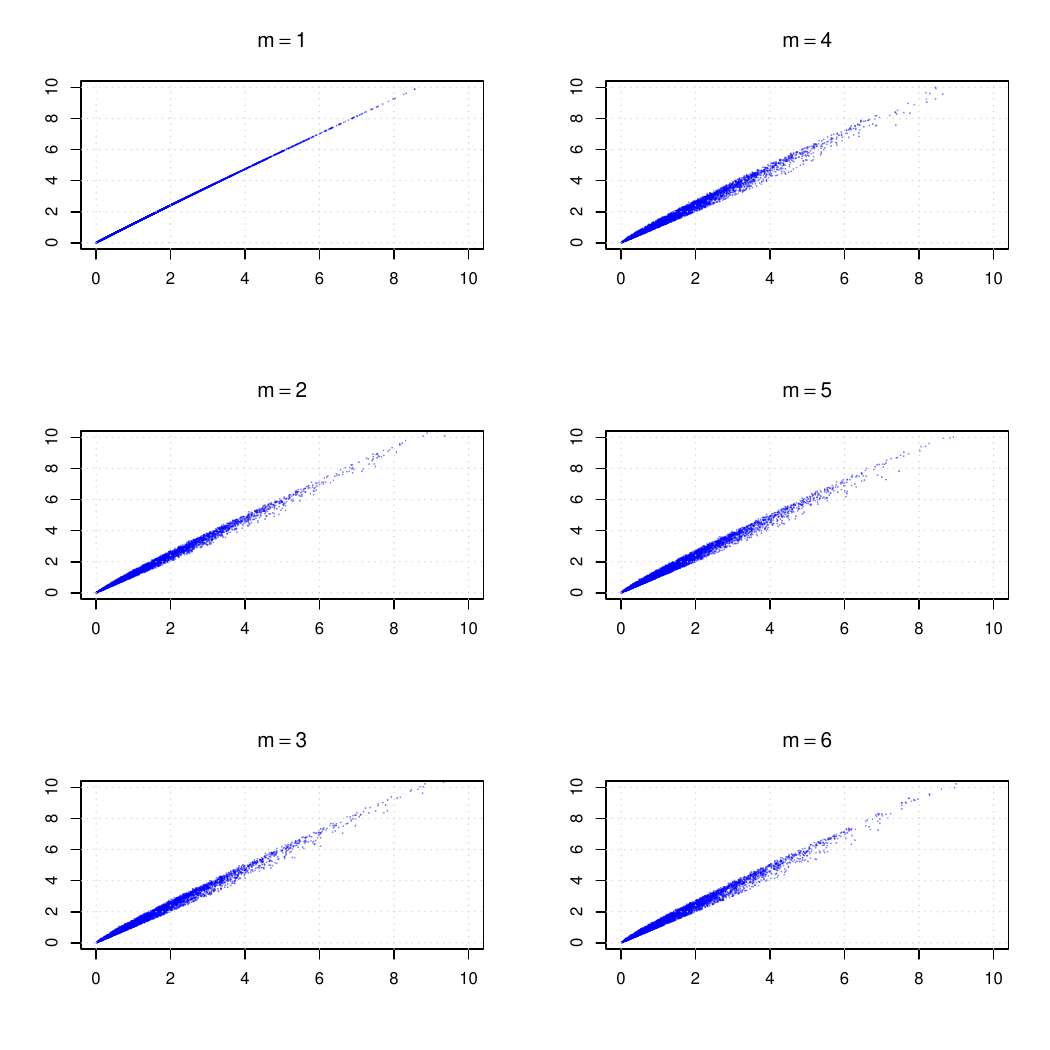}
\caption{5000 simulated pairs following the distribution $Q^{(2),m}_{t}((x_1,x_2),\cdot)$.}
 \label{fig:comonotone-xy}
\end{figure}

\begin{figure}
\centering
 \includegraphics[scale=0.75]{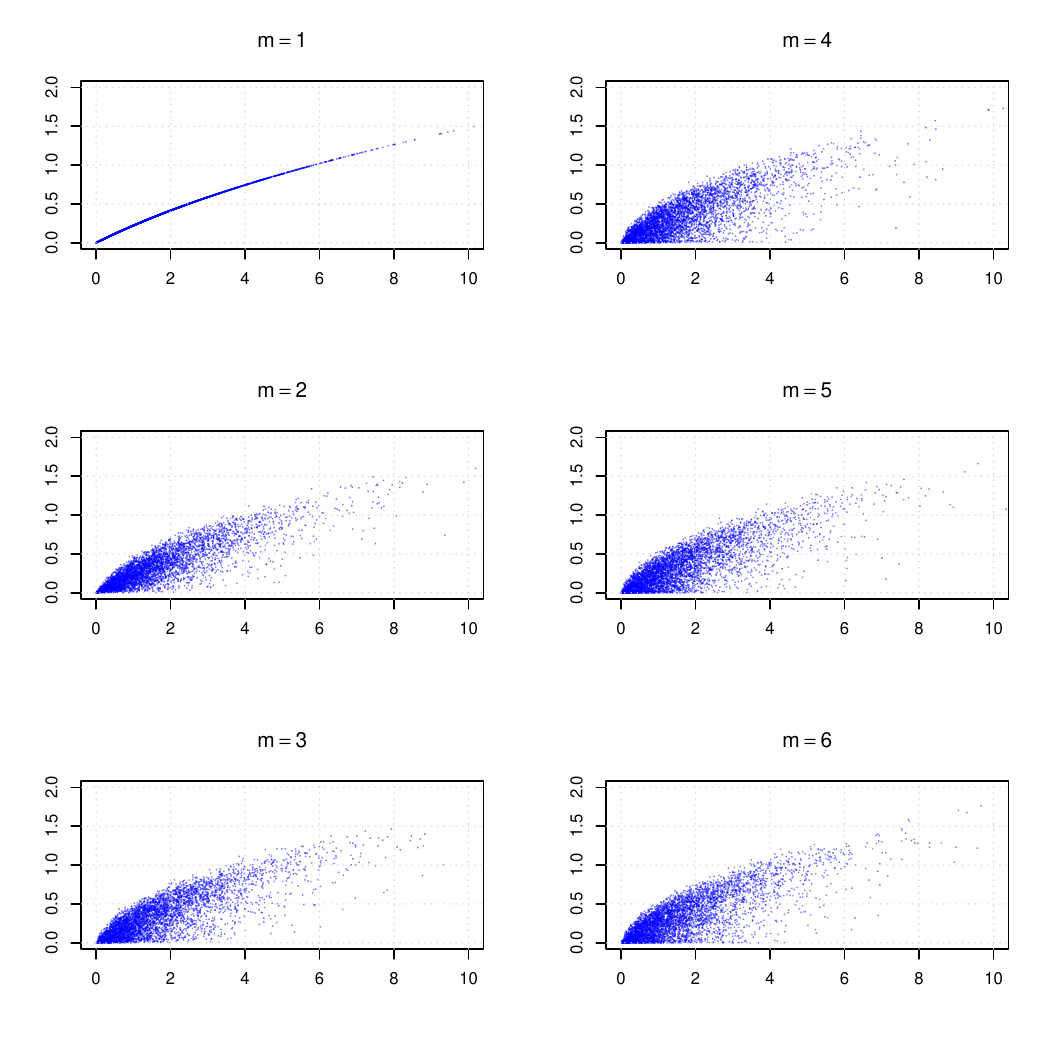}
\caption{Transformation $(x,y) \mapsto (x,y-x)$ applied to the graphs in Figure \ref{fig:comonotone-xy}.}
 \label{fig:comonotone-x-y-x}
\end{figure}

\end{appendix}


\end{document}